\documentclass{article}
\usepackage{graphicx} % Required for inserting images

\usepackage[english]{babel}
\usepackage{csquotes}
\usepackage[toc,page]{appendix}

\usepackage[
  a4paper,
  inner=2.2cm, outer=2.6cm,
  top=3cm,    bottom=3cm,
  bindingoffset=0.5cm,
  headsep=1em
]{geometry}
 
\usepackage{amsmath,amsthm}
\usepackage{mathtools}        % adds no math family, only macros
\usepackage{stix}             % text + math from STIX (Type 1)
 \usepackage{mathrsfs}
\usepackage{microtype}

\usepackage{graphicx}
\usepackage{xcolor}
\usepackage{booktabs}
\usepackage{array}
\usepackage{caption}
\usepackage{subcaption}
\usepackage{tikz}
\usepackage{enumitem}
\usepackage{bm}
\usepackage{xfrac}
\usetikzlibrary{calc,arrows.meta,decorations.pathmorphing}
\usepackage{csquotes}

\usepackage{titlesec}

\titleformat{\section}
  {\normalfont\large\scshape\filcenter}
  {\thesection.}{1em}{}

\titlespacing*{\section}
  {0pt}{2.5ex plus 1ex minus .2ex}{1.5ex plus .2ex}

\titleformat{\subsection}[runin]
  {\normalfont\bfseries}
  {\thesubsection.}{0.5em}{}[.]

\titlespacing*{\subsection}
  {0pt}{1.5ex plus 0.5ex minus 0.2ex}{1em}

\titleformat{\subsubsection}[runin]
  {\normalfont\itshape}
  {\thesubsubsection.}{0.5em}{}[.]

\titlespacing*{\subsubsection}
  {0pt}{1.2ex plus 0.4ex minus 0.2ex}{1em}

\titleformat{\chapter}[display]
  {\normalfont\huge\bfseries}
  {\chaptertitlename\ \thechapter}
  {20pt}
  {\Huge}

\titlespacing*{\chapter}{0pt}{50pt}{80pt}
 
\usepackage[
  colorlinks=true,
  linkcolor=black,
  citecolor=black,
  urlcolor=blue!60!black,
  bookmarksnumbered=true
]{hyperref}
\usepackage[nameinlink,capitalise,noabbrev]{cleveref}

\numberwithin{equation}{section}

\usepackage[
  backend=biber,
  style=alphabetic,
  sorting=nyt,
  maxnames=10,
  giveninits=true,
  doi=true, isbn=false, url=false
]{biblatex}
\usepackage{aliascnt}

\newcommand{\newshared}[3]{%
  \newaliascnt{#1}{thm}%
  \newtheorem{#1}[#1]{#2}%   <- [#1] aggiunto
  \aliascntresetthe{#1}%
  \crefname{#1}{#2}{#3}%
  \Crefname{#1}{#2}{#3}%
}

\theoremstyle{plain}
\newtheorem{thm}{Theorem}[section]
\crefname{thm}{Theorem}{Theorems}
\Crefname{thm}{Theorem}{Theorems}

\newshared{prop}{Proposition}{Propositions}
\newshared{lem}{Lemma}{Lemmas}
\newshared{cor}{Corollary}{Corollaries}

\newtheorem*{thm*}{Theorem}   % unnumbered, no aliascnt needed

\theoremstyle{definition}
\newshared{dfn}{Definition}{Definitions}
\newshared{ex}{Example}{Examples}
\newshared{prb}{Problem}{Problems}
\newshared{ass}{Assumption}{Assumptions}

\theoremstyle{remark}
\newshared{rmk}{Remark}{Remarks}
\newshared{notation}{Notation}{Notations}
 
\renewcommand{\div}{\operatorname{div}}

\newcommand{\tr}{\operatorname{tr}}

\newcommand{\dist}{\operatorname{dist}}
\newcommand{\diam}{\operatorname{diam}}

\DeclareMathOperator*{\esssup}{ess\,sup}
\DeclareMathOperator*{\essinf}{ess\,inf}

\newcommand{\loc}{\mathrm{loc}}

\newcommand{\R}{\mathbb{R}}

\newcommand{\Q}{Q}

\newcommand{\Mink}{\R^{m,1}}
\newcommand{\II}{\mathrm{II}}
\newcommand{\norm}[1]{\lVert#1\rVert}

\newcommand{\abs}[1]{\lvert#1\rvert}

\newcommand{\set}[1]{\{#1\}}

\newcommand{\inte}{\operatorname{int}}

\newcommand{\eps}{\varepsilon}
\newcommand{\scal}[2]{\langle #1,#2\rangle}
\newcommand{\Nu}{\mathrm{N}}
\newcommand{\Mu}{\mathrm{M}}
\newcommand{\Haus}{\mathscr{H}}

\newcommand{\babla}{\overline\nabla}

\newcommand{\diff}{d}
\newcommand{\de}{\mathrm d}
\newcommand{\del}{\partial}

\newcommand{\1}{\mathbb{1}}

\title{\textbf{Free boundary space-like graphs} \\ \textbf{with prescribed mean curvature}}
\author{%
  Lorenzo Maniscalco
}
\date{July 2026}

\begin{document}

\maketitle

\scriptsize \begin{center} Dipartimento di Matematica ``Giuseppe Peano'' \\ 
 Università degli Studi di Torino, Via Carlo Alberto 10, 10123 Torino, Italy\\
\texttt{lorenzo.maniscalco@unito.it}
\end{center}
\bigskip

\begin{abstract}
    We address the prescribed Lorentzian mean curvature problem over a convex bounded domain $\Omega$ of $\R^m$ with bounded right-hand side and homogeneous capillary boundary condition. We prove that the problem has a unique $W^{2,2}$-regular weak solution $u$ with zero mean and that $|Du| \leq 1 - \theta$ for some $\theta\in(0,1)$ only depending on the data. Such $u$ is also the unique maximizer of an associated functional. A key step in proving that the maximizer is a weak solution consists in showing that it has no light segments, i.e. segments along which $|Du| = 1$. This holds for arbitrary bounded capillary boundary data and can thus be an interesting result on its own.
\end{abstract}


\section{Introduction}

\subsection*{Main result} In this paper we investigate existence and regularity of solutions to the following oblique derivative problem on a bounded Lipschitz domain $\Omega \subset\R^m$
\begin{align}\label{eq:problem_Rm}
        -\div\Bigg(\dfrac{Du}{\sqrt{1 - |Du|^2}}\Bigg)= \rho \quad \text{in $\Omega$,} \qquad \qquad
        \dfrac{Du \cdot n}{\sqrt{1 - |Du|^2}} = \psi \quad \text{on $\del\Omega$},
\end{align}
where $n$ is the exterior unit normal of $\del\Omega$ and $\rho$ and $\psi$ are given functions on $\Omega$ and $\del\Omega$ respectively. A \emph{weak solution} to \eqref{eq:problem_Rm} is a function $u \in W^{1,\infty}(\Omega)$ such that
% \begin{align*}
%     w_u \doteq \frac{1}{\sqrt{1 - |Du|^2}} \in L^1(\Omega)
% \end{align*}
% and
\begin{align}
    w_u &\doteq \frac{1}{\sqrt{1 - |Du|^2}} \in L^1(\Omega),\label{eq:energy} \\ \int_\Omega w_u Du \cdot D\eta \, dx &= \int_\Omega \rho \eta \, dx + \int_{\del\Omega} \psi\eta \, d\Haus^{m-1}\label{eq:Rm_weak_sol}
\end{align}
for every $\eta \in C^1(\overline\Omega)$. For reasons that will be explained shortly, the function $w_u$ is called \emph{tilt function} or \emph{energy density} of $u$. Testing \eqref{eq:Rm_weak_sol} with $\eta = 1$ gives a necessary condition on $\rho$ and $\psi$  for the existence of a solution
\begin{align}\label{eq:compatibility}
    \int_\Omega \rho \, dx + \int_{\del\Omega} \psi \, d\Haus^{m-1} = 0.
\end{align}
Our result deals with the case of convex domains and homogeneous boundary conditions.

\begin{thm}\label{thm:Rm_main}
    Let $\Omega\subseteq \R^m$ be a bounded convex domain and let $\rho$ be a zero-mean bounded function on $\Omega$. If $\psi = 0$, then problem \eqref{eq:problem_Rm} has a weak solution in $W^{2,2}(\Omega)\cap C(\overline\Omega)$. Such solution is unique up to additive constants and, if $|\rho|\leq \Lambda$ on $\Omega$ for some $\Lambda\geq 0$, there exists a positive constant $C = C(m,\Omega,\Lambda)$ such that $w_u \leq C$.
\end{thm}

The Dirichlet problem on bounded domains for the operator in \eqref{eq:problem_Rm} was first addressed by Bartnik and Simon in \cite{BartnikSimon1982} for a bounded $\rho$. In a series of papers \cite{KlyMik92, KlyMik93, Klyachin03} Klyachin and Miklyukov investigated the case where $\rho$ is a finite sum of Dirac deltas, and more recently Byeon, Ikoma, Malchiodi and Mari \cite{BIMM} addressed the case of more general measures. To the best of our knowledge, \cref{thm:Rm_main} is the first existence and regularity result for a non-Dirichlet boundary problem for this operator. 

\subsection*{Relevance in Physics and Geometry} The operator in \eqref{eq:problem_Rm} arises naturally in non-linear electrostatics and Lorentzian geometry. According to the Born--Infeld model of Electromagnetism -- a non-linear theory of the electromagnetic field developed by M. Born and L. Infeld in the 30s to overcome the infinite self-energy of the static point charge \cite{Born33,BornInfeld33,BornInfeld34} -- the electrostatic potential $u$ generated by a charge $\rho$ solves
\begin{align}\label{eq:mean_curva_op_Rm}
    H_u \doteq \div\left( \frac{Du}{\sqrt{1 - |Du|^2}}\right) = -\rho.
\end{align}
% A major feature of this model is that it provides for an a priori bound on the intensity of the electrostatic field: $|\mathbf E| < 1$.
The interpretation of $\rho$ as an electrostatic charge makes it natural to investigate the problem when $\rho$ is measure, as it has been done extensively in the literature.

It can be shown that the energy density of the electric field $-Du$ generated by $\rho$ is $w_u - 1 + u\rho$. Hence, if $\rho\in L^2(\Omega)$, the condition \eqref{eq:energy} is equivalent to the electrostatic energy of the field being finite. For this reason, as anticipated, the function $w_u$ is sometimes called, with a slight abuse of notation, the \emph{energy density function} of $u$. Complete accounts of the Born--Infeld model can be found in \cite{BialynickiBirula83}, \cite{Yang2000} and \cite{Kiessling} (see also \cite{BonPompDav} and the Appendix of \cite{BIMM}). The Born--Infeld equation was also shown to be relevant in string theory (see for instance \cite{Gibbons97}, \cite{Yang2000}).

\cref{thm:Rm_main} is also an existence and regularity result for capillary space-like  graphs in a cylinder of Minkowski spacetime $\Mink$ with prescribed mean curvature. In Cartesian coordinates, the Minkowski scalar product writes $\scal{\cdot}{\cdot} = - (dx^0)^2 + (dx^1)^2 + \dots + (dx^m)^2$. An embedded smooth hypersurface $\Sigma$ is \emph{strictly space-like} if $\scal{\cdot}{\cdot}$ restricts to a Riemannian metric on $\Sigma$. If $\Sigma$ is the graph of a function $u:\Omega\to\R$ on some $\Omega\subseteq\R^m$, the space-like condition is equivalent to $|Du| < 1$. A space-like graph admits a global unit normal given by
\begin{align*}
    N = w_u(Du + \del_0)
\end{align*}
and its mean curvature in direction $N$ is $H_u$, the differential operator defined in \eqref{eq:mean_curva_op_Rm}. The boundary condition in \eqref{eq:problem_Rm} can be interpreted as a capillary condition, that is, a prescription of the contact angle between $\Sigma$ and the boundary of the cylinder $\R\times \Omega$. Indeed, if we identify $n$ with its horizontal extension $(0,n)$ along $\R \times \del\Omega$, the contact angle $\gamma$ is defined by
\begin{align}\label{eq:contact_angle}
    \sinh\gamma \doteq - \scal{N}{n} = - w_u Du \cdot n.
\end{align}
If $\gamma = 0$ we say that $\Sigma$ is a \emph{free boundary} hypersurface.

Notice also that, since $\scal{\del_0}{\del_0} = - 1$, the function $w_u$ can be interpreted as the hyperbolic cosine of the angle $\beta$ between $N$ and the vertical direction $\del_0$:
\begin{align*}
    \cosh \beta \doteq - \scal{N}{\del_0} = w_u.
\end{align*}
Such angle tends to infinity as the space-like condition fails, i.e. as $|Du| \to 1$. The function $w_u$ is also called the \emph{tilt function} of (the graph of) $u$ and the bound $w_u \leq C$ established by \cref{thm:Rm_main} can be interpreted as the \emph{uniform space-likeness} of (the graph of) $u$.

% \begin{dfn}
%     Let $\Omega$ be a bounded Lipschitz domain. A \emph{weak solution} to \eqref{eq:problem_Rm} is a weakly space-like function $u$ such that $w \in L^1(\Omega)$ and 
%     \begin{align}
%         \int_\Omega wDu \cdot D\eta \, dx = \int_\Omega \eta\rho \, dx + \int_{\del\Omega} \eta\psi \, d\Haus^{m-1}, \qquad \forall\eta \in C^1(\overline\Omega).
%     \end{align}
% \end{dfn}

% Testing with $\eta \equiv 1$ gives the \emph{compatibility condition}
% \begin{align}
%     \int_\Omega \rho + \int_{\del\Omega} \psi = 0,
% \end{align}
% which is necessary for the existence of weak solutions to \eqref{eq:problem_Rm}

\subsection*{A no-light-segment result}

Problem \eqref{eq:problem_Rm} appears as the Euler--Lagrange equation of the functional
\begin{align}\label{eq:Rm_functional}
    I_{\rho,\psi}(u) \doteq \int_\Omega \Big(\sqrt{1 -|Du|^2} + \rho u \Big) \, dx + \int_{\del\Omega} \psi u \, d\Haus^{m-1}
\end{align}
defined on the set of $1$-Lipschitz functions on $\Omega$. The density $w_u^{-1} =\sqrt{1 - |Du|^2}$ is the volume element of the graph of $u$ with respect to the metric induced by the immersion in $\Mink$. The functional  $I_{\rho,\psi}$ is easily seen (see \cref{lem:existence_maximizer} below) to have a unique maximizer in the class of $1$-Lipschitz functions with zero mean
% -- that we will refer to as the \emph{solution of the variational problem}. 
However, due to the lack of smoothness of volume element as $|Du| = 1$, it is not {a priori} guaranteed that such maximizer would be a weak solution to \eqref{eq:problem_Rm}.
% Most of the effort is devoted to prove that the maximizer is uniformly space-like (i.e. $w_u \in L^\infty(\Omega)$ or $|Du| \leq 1 - \theta$ for some $\theta \in (0,1)$), which suffice to show that the maximizer is a weak solution. The proof of \cref{thm:Rm_main} follows the scheme of the proof of \cite[Theorem 4.1]{BartnikSimon1982} and relies on three following facts.
% \subsubsection*{1 . Absence of light segments}
A major fact that could keep the maximizer of $I_{\rho,\psi}$ from being a weak solution to \eqref{eq:problem_Rm} is the possible presence of \emph{light segments}.

\begin{dfn}
    For every $x,y \in \R^m$ let $\overline{xy}\doteq \set{sx + (1-s)y \ | \ s\in[0,1]}$. A segment $\overline{xy}$ is a \emph{light segment} for a $1$-Lipschitz function $u$ if $\overline{xy}\subseteq\overline\Omega$ and $u(y) - u(x) = |y - x|$.
\end{dfn}

As \cite[Corollary 1.10]{BIMM} shows, the presence of light segments for the maximizer, other than being a technical difficulty, can be a genuine obstruction for the latter to be a weak solution. Our no-light-segment Theorem holds for any bounded boundary prescription $\psi$, hence could prove useful also in investigating extensions of \cref{thm:Rm_main} to non-homogeneous capillary conditions.

\begin{thm}\label{thm:capillary_no_light}
    Let $\Omega\subseteq\R^m$ be a bounded, convex domain. For fixed $\rho \in L^\infty(\Omega)$ and $\psi \in L^\infty(\del\Omega)$, let $u$ be the maximizer of $I_{\rho,\psi}$ among the $1$-Lipschitz functions of zero mean. Then $u$ does not have light segments in the interior of $\Omega$.
\end{thm}

Actually we will prove a stronger result that deals also with mixed Dirichlet and capillary boundary conditions, see \cref{thm:stronger_no_light}.

\subsection*{Sketch of the proof of \cref{thm:Rm_main} and open questions}

It will be sufficient to show that the maximizer $u$ is $W^{2,2}$-regular and uniformly space-like, i.e. $|Du| \leq 1 - \theta$. Up to approximating from the interior, we can assume $\Omega$ smooth and strictly convex. Let $\rho_j \to \rho$ be an approximation of $\rho$ by standard mollification, so that $\norm{\rho_j}_{L^\infty(\Omega)} \leq 2\Lambda$. It is not difficult to show that a sequence $u_j$ of solutions to problem \eqref{eq:problem_Rm} with smooth right-hand side $\rho_j$ and $\psi = 0$ would converge uniformly to $u$ (see \cref{lem:converging_of_maximizers} below). The existence of such solutions follows from an a priori pointwise gradient estimate (\cref{thm:capillary_grad_est}) and the classical theory of Lieberman for oblique derivative problems \cite{Lieberman88} \cite{Lieberman2013}. Our gradient estimate degenerates as the $C^1$ norm of $\rho_j$ grows, hence it is not useful in providing uniform space-likeness. This will instead follow from our no-light-segment \cref{thm:capillary_no_light} combined with the \emph{monotonicity formula}
\begin{align}\label{eq:monotonicity_intro}
    C\big(1 - |Du_j(o)|^2\big)^{\alpha/2} \geq \frac{1}{R^m} \int_\Omega \big(1 - |Du_j|^2\big)^{(\alpha + 1)/2} \, dx + \frac{1}{R^{m-2}}\int_\Omega |D^2u_j|^2 \, dx, \qquad \forall o \in \Omega,
\end{align}
which holds for every $R \geq \diam \Omega$, with $C = C(m,R,\Lambda)$ and $\alpha = \alpha(m) < m^{-1}$ (cf. \cref{lem:BS_monotonicity}). This is an improvement of the monotonicity formula \cite[Lemma 2.1]{BartnikSimon1982} of Bartnik and Simon. Estimate \eqref{eq:monotonicity_intro} gives a uniform $W^{2,2}$ estimate along $\set{u_j}$, which in turn shows $u \in W^{2,2}(\Omega)$. Moreover, passing to the limit, it implies that, if the maximizer is light-like at a point $o\in\Omega$, i.e. $|Du|(o) = 1$, then it must be light-like in the whole of $\Omega$, in stark contradiction with \cref{thm:capillary_no_light}. This suffices to show that $u$ must be uniformly space-like and \eqref{eq:energy}-\eqref{eq:Rm_weak_sol} follow.

What keeps us from an existence result on general Lipschitz domains and more general capillary data is the lack of an a priori gradient estimate that generalizes \cref{thm:capillary_grad_est}. Our estimates \cref{lem:BS_monotonicity} and \cref{thm:capillary_grad_est} rely heavily on both the convexity and the free boundary condition. Classical approaches developed for the capillary problem in Euclidean space, such as the integral iteration approach of Simon and Spruck \cite{SimonSpruck} or the geometric one of Korevaar \cite{Korevaar}, work on more general domains and capillary data but, to the best of our understanding, do not adapt to the Lorentzian geometry. Establishing similar gradient bounds seems a challenging problem.

\subsection*{Acknowledgements} The author is grateful to A. Boscaggin for suggesting the problem, to A. Iacopetti and L. Mari for carefully reading the manuscript and to R. Ziegele for helpful discussions.

\section{The variational problem}\label{sec:variational}

\subsection{The setting}

Let $\Omega\subseteq \R^m$ a bounded domain. A function $u:\Omega \to \R$ is called
\begin{itemize}
    \item \emph{weakly space-like} if $u \in W^{1,\infty}(\Omega)$ and $\norm{Du}_{L^\infty(\Omega)}\leq 1$; we let $A(\Omega)$ be the set of such functions\footnote{$A$ stands for \emph{achronal}, a term from Causality Theory};
    \item \emph{space-like} if for every $x,y\in\Omega$ such that $\overline{xy}\subset\overline\Omega$ it holds $|u(x) - u(y)| < |x - y|$; we will denote by $S(\Omega)$ the set of space-like functions;
    \item \emph{strictly space-like} if $u \in C^1(\Omega)$ and $|Du| < 1$ in $\Omega$; 
    \item \emph{uniformly space-like} if $u$ is weakly space-like and its tilt function $w_u \doteq (1 - |Du|^2)^{-1/2}$ is an element of $L^\infty(\Omega)$.
\end{itemize}
The terminology comes from Relativity. Accordingly, we will say that $u$ \emph{goes null} (or that it is \emph{light-like}) where $|Du| = 1$.

If $\varphi\in C(\del\Omega)$ we let $A_\varphi(\Omega)$ and $S_\varphi(\Omega)$ be the sets of weakly space-like or space-like functions whose trace on $\del\Omega$ is $\varphi$. More generally, if $\Gamma$ is a compact subset of $\del\Omega$ and $\varphi\in C(\del\Omega)$, we let
\begin{align*}
    A_\varphi(\Omega;\Gamma) &\doteq \set{u \in A(\Omega) \ | \ u|_\Gamma = \varphi|_\Gamma}, \qquad A_\varphi(\Omega) \doteq A_\varphi(\Omega;\del\Omega), \\
    S_\varphi(\Omega;\Gamma) &\doteq \set{u \in S(\Omega) \ | \ u|_\Gamma = \varphi|_\Gamma}, \qquad S_\varphi(\Omega) \doteq S_\varphi(\Omega;\del\Omega).
\end{align*}
Notice that, if $\Gamma = \varnothing$, then $A_\varphi(\Omega;\Gamma) = A(\Omega)$. The space $A_\varphi(\Omega;\Gamma)$ is not empty provided that $\varphi$ satisfies
\begin{align}\label{eq:hp_varphi}
    |\varphi(x) - \varphi(y)| \leq |x - y| \qquad \qquad \text{$\forall x,y \in \Gamma$ such that $\overline{xy}\subset\overline\Omega$}.
\end{align}
Notice that this condition is empty if $\Gamma = \varnothing$. Finally, we use the symbol $\ \mathring{} \ $ to signify that a set is made of zero-mean functions, for example
\begin{align*}
    \mathring A_\varphi(\Omega;\Gamma) \doteq \bigg\{u \in A_\varphi(\Omega;\Gamma) \ | \ \int_\Omega u \, dx = 0 \bigg\}.
\end{align*}
For given $\rho \in L^2(\Omega)$ and $\varphi\in C(\del\Omega)$ consider the functional
\begin{align*}
    I_\rho: A_\varphi(\Omega;\Gamma) \to \R, \qquad I_\rho(u) \doteq \int_\Omega \sqrt{1 - |Du|^2} + \int_\Omega\rho u.
\end{align*}
The functional $u\mapsto \int_\Omega\sqrt{1 - |Du|^2} \, dx$ is called \emph{area functional}, as it measures the intrinsic area of the graph of $u$ as a hypersurface in $\Mink$. Now take $\psi\in L^2(\Nu)$ and consider the functional
\begin{align*}
    I_{\rho,\psi}: A_\varphi(\Omega;\Gamma) \to \R, \qquad I_{\rho,\psi}(u)\doteq I_\rho(u) + \int_\Nu \psi u.
\end{align*}

\begin{prb}\label{prb:variational_problem}
    Let $\Omega\subseteq\R^m$ be a bounded domain, $\Gamma\subseteq\del\Omega$ a compact set and let $\Nu \doteq \del\Omega\setminus\Gamma$. Let $\varphi\in C(\del\Omega)$ such that \eqref{eq:hp_varphi} holds and fix $\rho \in L^2(\Omega)$ and $\psi \in L^2(\del\Omega)$ such that the compatibility condition \eqref{eq:compatibility} holds true if $\Gamma = \varnothing$. The variational problem we address in this Section is the following
    \begin{align}
    \max_{u \in X(\Omega)}I_{\rho,\psi}(u), \qquad \qquad X(\Omega)\doteq
        \begin{cases}
            A_\varphi(\Omega;\Gamma) & \text{if } \Gamma \neq \varnothing, \\
            \mathring A(\Omega) & \text{if } \Gamma = \varnothing.
        \end{cases}
    \end{align}
\end{prb}

The formal Euler--Lagrange equations of $I_{\rho,\psi}$ are
\begin{align}\label{eq:general_problem}
    - H_u = \rho \quad \text{in $\Omega$}, \qquad u = \varphi \quad \text{on $\Gamma$}, \qquad w_u Du \cdot n = \psi \quad \text{on $\Nu$},
\end{align}
and we say that $u$ is a weak solution to \eqref{eq:general_problem} if $\int_{\Omega'}w_u < \infty$ for every $\Omega' \Subset \overline\Omega\setminus\Gamma$ and \eqref{eq:Rm_weak_sol} holds for every $\eta \in C_c(\overline{\Omega}\setminus \Gamma)$.

\subsection{Properties of the maximizer} Unless stated differently, throughout the Section $\Omega$, $\Gamma$, $\Nu$, $\varphi$ and $\psi$ will be as in \cref{prb:variational_problem}. First of all we establish existence and uniqueness of the maximizer of $I_{\rho,\psi}$. 

% First observe that by the concavity of $p \mapsto \sqrt{1 - |p|^2}$ we deduce the following

% \begin{lem}\label{lem:area_semicontinuous}
%     The area functional $I_{0,0}:A(\Omega) \to \R$ is strictly convex and upper semicontinuous with respect to the uniform convergence in $\Omega$.
% \end{lem}

% The example of a highly oscillating, nearly light-like, sequence of function converging uniformly to $0$ shows that the area functional is not continuous.

\begin{lem}\label{lem:existence_maximizer}
    \cref{prb:variational_problem} has a unique solution $u\in X(\Omega)$.
\end{lem}
\begin{proof}
    If $\Gamma\neq\varnothing$, fix $x_0\in\Gamma$; every $u\in X(\Omega)$ satisfies $u(x_0)=\varphi(x_0)$, so
    \begin{align*}
        |u(x)| \leq |u(x)-u(x_0)| + |\varphi(x_0)| \leq \de(x,x_0) + |\varphi(x_0)| \leq \diam\Omega + \norm{\varphi}_{C(\del\Omega)}, \quad x\in\overline\Omega.
    \end{align*}
    If $\Gamma=\varnothing$, take $u\in\mathring A(\Omega)$. Since it has zero mean, it must vanish at some point $x_0 \in \Omega$, and since it is $1$-Lipschitz we must have $|u(x)| = |u(x) - u(x_o)| \leq |x - x_0| \leq \diam\Omega$. In either case $X(\Omega)$ is equibounded in $C(\overline\Omega)$. By space-likeness, it is also equicontinuous, hence precompact in $C(\overline\Omega)$ by Ascoli--Arzelà. It is also closed: uniform limits of $1$-Lipschitz functions are $1$-Lipschitz, traces pass to the limit under uniform convergence, and $u\mapsto\int_\Omega u$ is continuous in $C(\overline\Omega)$. Thus $X(\Omega)$ is compact in $C(\overline\Omega)$.
    % Notice also that $X(\Omega)$ is also convex and non empty, as it always contains
    % \begin{align*}
    %     u_0(x) \doteq \inf_{y\in\Gamma} (\varphi(y) + |x - y|), \qquad u_0 = 0 \text{ if $\Gamma=\varnothing$}.
    % \end{align*}
    Using the uniform bound on $\norm{u}_{L^\infty(\Omega)}$ obtained above and the boundedness of $\Omega$, the functional satisfies
    \begin{align*}
        I_{\rho,\psi}(u) \leq |\Omega| + \norm{\rho}_{L^2(\Omega)}\norm{u}_{L^2(\Omega)} + \norm{\psi}_{L^2(\Nu)}\norm{u}_{L^2(\Nu)},
    \end{align*}
    thus, it is bounded above on $X(\Omega)$. %since $\norm{u}_{L^1(\Omega)}\le |\Omega|\norm{u}_{L^\infty(\Omega)}$ and $\norm{u}_{L^2(\Nu)}\le \mathscr H^{m-1}(\Nu)^{1/2}\norm{u}_{L^\infty(\Omega)}$.

    Let $\{u_j\}\subset X(\Omega)$ be a maximizing sequence. By compactness, up to a subsequence $u_j\to u$ in $C(\overline\Omega)$ for some $u\in X(\Omega)$. Traces converge in $L^2(\Nu)$ by continuity of the trace map under uniform convergence. Since the map $p \mapsto \sqrt{1 - |p|^2}$ is strictly concave, the Lorentzian area is strictly concave as well, and thus, upper-semicontinuous,
    \begin{align*}
        \sup_{X(\Omega)} I_{\rho,\psi} = \lim_{j\to\infty} I_{\rho,\psi}(u_j) \leq \int_\Omega\sqrt{1-|Du|^2} + \int_\Omega u\rho + \int_\Nu u\psi = I_{\rho,\psi}(u),
    \end{align*}
    hence, $u$ is a maximum point.

    By strict concavity, two maximizers must have the same gradient a.e. on $\Omega$, hence must differ by a constant. By definition of $X(\Omega)$, this implies that the two are actually equal.
    % Now let $u,v\in X(\Omega)$ both be maximizers. Since $X(\Omega)$ is convex, $h\doteq\tfrac12(u+v)\in X(\Omega)$. If $Du\neq Dv$ on a set of positive measure, strict concavity of $p\mapsto\sqrt{1-|p|^2}$ and linearity of the remaining terms give
    % \begin{align*}
    %     \sup_{X(\Omega)} I_{\rho,\psi} = \tfrac12 I_{\rho,\psi}(u) + \tfrac12 I_{\rho,\psi}(v) < I_{\rho,\psi}(h),
    % \end{align*}
    % a contradiction. Hence $Du=Dv$ a.e.\ in $\Omega$, and since $\Omega$ is connected, $u-v\equiv c\in\R$. If $\Gamma\neq\varnothing$, $c=0$ since $u=v=\varphi$ on $\Gamma$; if $\Gamma=\varnothing$, $c=0$ since $\int_\Omega u=\int_\Omega v=0$. In both cases $u=v$.
\end{proof}

\begin{rmk}\label{rmk:solutions_are_max}
    Notice that every weak solution to \eqref{eq:general_problem} is a solution of the variational \cref{prb:variational_problem}. This follows from the fact that being a weak solution is equivalent to $I_{\rho,\psi}$ being differentiable at $u$ with $I_{\rho,\psi}'(u) = 0$ and by the strict concavity of $I_{\rho,\psi}$.
\end{rmk}

Next we observe that the property of being a maximizer is stable under restrictions of the domain.

\begin{lem}\label{lem:local_minimizer}
    Let $u$ be the solution of \cref{prb:variational_problem} and fix a subdomain $\Omega'\subset \Omega$.
    % assume that $\rho \in L^2(\Omega)$, $\psi\in L^2(\del\Omega)$ are compatible and let $u$ be the maximizer of $I_{\rho,\psi}$ in $\mathring A(\Omega)$.
    Then $u|_{\Omega'}$ is the unique maximizer of $I_{\rho,\psi}$ in $A_u(\Omega',\overline\Omega\cap \del\Omega')$.
\end{lem}
\begin{proof}
    Set $\Gamma'\doteq(\Gamma \cup {\Omega})\cap \del\Omega'$, $\Nu' = \del\Omega' \setminus \Gamma'\subseteq\Nu$ and let us use the shortcuts $I = I_{\rho,\psi}$, $I' = I|_{A_u(\Omega',\Gamma')}$ and $u' = u|_{\Omega'}$. Suppose there were $v' \in A_u(\Omega',\Gamma')$ such that $I'(v')> I'(u')$. Define 
    \begin{align*}
        v(x) = 
        \begin{cases}
            v'(x) + c &\quad \text{on $\Omega'$} \\
            u(x) + c &\quad \text{on $\Omega\setminus\Omega'$}
        \end{cases}
    \end{align*}
    where $c=\frac{1}{|\Omega|}\int_{\Omega'}(u - v')$ is chosen so that $v\in \mathring A(\Omega)$. Then by the compatibility condition \eqref{eq:compatibility} and the maximality of $v'$ we have
    \begin{align*}
        I(v) &= I'(v') + \int_{\Omega\setminus\Omega'} \sqrt{1 - |Du|^2} + \int_{\Omega\setminus\Omega'}u\rho + \int_{\Nu \setminus\Nu'} u\psi + c\left( \int_\Omega \rho + \int_{\del\Omega}\psi\right) \\
        &> I'(u') + \int_{\Omega\setminus\Omega'} \sqrt{1 - |Du|^2} + \int_{\Omega\setminus\Omega'}u\rho + \int_{\Nu \setminus\Nu'} u\psi = I(u),
    \end{align*}
    a contradiction with $u$ being a maximizer of $I$. It follows that $I'(u')\geq I'(v')$ for each $v'\in A_u(\Omega',\Gamma')$.
\end{proof}

We establish a comparison principle for maximizers.

\begin{lem}\label{lem:comparison}
    Fix $\varphi_1,\varphi_2\in C(\del\Omega)$ such that \eqref{eq:hp_varphi} holds true and let $X_i(\Omega)$ as in \cref{prb:variational_problem}. For $\psi_1,\psi_2\in L^\infty(\Nu)$ and $\rho_1,\rho_2\in L^\infty(\Omega)$ consider
    \begin{align*}
        I_{\rho_i,\psi_i}: X_i(\Omega) \to \R, \qquad I_{\rho_i,\psi_i}(v) \doteq \int_\Omega\sqrt{1-|Dv|^2} + \int_\Omega v\rho_i + \int_\Nu v\psi_i.
    \end{align*}
    If $u_i$ is the maximizer of $I_{\rho_i,\psi_i}$ in $X_i(\Omega)$, then
    \begin{align}\label{eq:comparison}
        \rho_2\leq\rho_1, \ \psi_2\leq\psi_1 \ \text{a.e.} \quad\Longrightarrow\quad u_2 - u_1 \leq \sup_\Gamma(\varphi_2-\varphi_1)
    \end{align}
    where the supremum over $\Gamma$ is set to zero in case $\Gamma = \varnothing$.
\end{lem}
\begin{proof}
    Set $s\doteq\sup_\Gamma(\varphi_2-\varphi_1)$, well-defined and finite since $\varphi_1,\varphi_2\in C(\Gamma)$ and $\Gamma$ is compact. Let $s = 0$ if $\Gamma = \varnothing$. Write $I_i \doteq I_{\rho_i,\psi_i}$. Assume first that $\Gamma \neq \varnothing$.

    Set $\tilde u_2\doteq u_2-s$, so that $\tilde u_2\in A_{\varphi_2-s}(\Omega;\Gamma)$. For $v\in A_{\varphi_2-s}(\Omega;\Gamma)$ we have $v+s\in A_{\varphi_2}(\Omega;\Gamma)$, hence the maximality of $u_2$ over $A_{\varphi_2}(\Omega;\Gamma)$ gives
    \begin{align*}
        I_2(v) = I_2(v+s) - s\Big(\int_\Omega\rho_2+\int_\Nu\psi_2\Big) \leq I_2(u_2) - s\Big(\int_\Omega\rho_2+\int_\Nu\psi_2\Big) = I_2(\tilde u_2)
    \end{align*}
    for every $v\in A_{\varphi_2-s}(\Omega;\Gamma)$, i.e.\ $\tilde u_2$ maximizes $I_2$ over $A_{\varphi_2-s}(\Omega;\Gamma)$.

    By definition of $s$, $\varphi_2-s\leq\varphi_1$ pointwise on $\Gamma$, i.e.\ $\tilde u_2\leq u_1$ on $\Gamma$. Set
    \begin{align*}
        u_+\doteq\max(u_1,\tilde u_2), \qquad u_-\doteq\min(u_1,\tilde u_2).
    \end{align*}
    Since $u_1,\tilde u_2\in W^{1,\infty}(\Omega)$, so are $u_+,u_-$, and by Stampacchia's theorem $Du_1=D\tilde u_2$ a.e.\ on $E\doteq\{u_1=\tilde u_2\}$; writing $E_+\doteq\{u_1>\tilde u_2\}$, $E_-\doteq\{\tilde u_2>u_1\}$, we get a.e.\ in $\Omega$
    \begin{align*}
        Du_+ = Du_1\,\mathbb 1_{E_+} + D\tilde u_2\,\mathbb 1_{E_-} + Du_1\,\mathbb 1_E, \qquad Du_- = D\tilde u_2\,\mathbb 1_{E_+} + Du_1\,\mathbb 1_{E_-} + Du_1\,\mathbb 1_E,
    \end{align*}
    so $\norm{Du_+}_{L^\infty(\Omega)},\norm{Du_-}_{L^\infty(\Omega)}\leq1$. Moreover $u_+ \in A_{\varphi_1}(\Omega;\Gamma)$ and $u_-\in A_{\varphi_2-s}(\Omega;\Gamma)$. From the decomposition of the gradients above, the area part of the functional is preserved, i.e.
    \begin{align*}
        I_{0,0}(u_+) + I_{0,0}(u_-) = I_{0,0}(u_1) + I_{0,0}(u_2)
        % \int_\Omega\sqrt{1-|Du_+|^2} + \int_\Omega\sqrt{1-|Du_-|^2} &= \int_{E_+}\big(\sqrt{1-|Du_1|^2}+\sqrt{1-|D\tilde u_2|^2}\big) \ + \\ + \int_{E_-}\big(\sqrt{1-|D\tilde u_2|^2}+\sqrt{1-|Du_1|^2}&\big) 
        %  + \int_E\big(\sqrt{1-|Du_1|^2}+\sqrt{1-|D\tilde u_2|^2}\big) \\
        % &= \int_\Omega\sqrt{1-|Du_1|^2} + \int_\Omega\sqrt{1-|D\tilde u_2|^2}.
    \end{align*}
    For any $p,q\in L^\infty(\Omega)$ and any measurable $a,b$, a direct check on $E_+,E_-,E$ gives the pointwise identity
    \begin{align*}
        \big(\max(a,b)\,p+\min(a,b)\,q\big) - (ap+bq) = (b-a)(p-q)\,\mathbb 1_{\{b>a\}} \qquad \text{a.e.}
    \end{align*}
    Applying this with $a=u_1,b=\tilde u_2$, $p=\rho_1,q=\rho_2$ on $\Omega$, and $p=\psi_1,q=\psi_2$ on $\Nu$, and using $\rho_1-\rho_2\geq0$, $\psi_1-\psi_2\geq0$ together with $\tilde u_2-u_1>0$ on $E_-$:
    \begin{align*}
        \int_\Omega(u_+\rho_1+u_-\rho_2) - \int_\Omega(u_1\rho_1+\tilde u_2\rho_2) &= \int_{E_-}(\tilde u_2-u_1)(\rho_1-\rho_2) \geq 0, \\
        \int_\Nu(u_+\psi_1+u_-\psi_2) - \int_\Nu(u_1\psi_1+\tilde u_2\psi_2) &= \int_{E_-\cap\Nu}(\tilde u_2-u_1)(\psi_1-\psi_2) \geq 0.
    \end{align*}
    Adding the area identity and the two inequalities above,
    \begin{align*}
        I_1(u_+) + I_2(u_-) \geq I_1(u_1) + I_2(\tilde u_2).
    \end{align*}
    On the other hand the maximality of $u_1$ and of $\tilde u_2$ give the reversed inequality, hence the above is actually an equality. In particular $I_1(u_+) = I_1(u_1)$ and by uniqueness of the maximizer it must be $u_+ = u_1$, that is $\tilde u_2 \leq u_1$, which is the claim.

    In case $\Gamma = \varnothing$ the proof works verbatim with $s = 0$ and $\mathring A(\Omega)$ in place of $A_{\varphi}(\Omega;\Gamma)$.
\end{proof}

Finally, we observe that the maximizers of converging data converge to the maximizer of the limit functional.

\begin{lem}\label{lem:converging_of_maximizers}
    Let $\Omega \subseteq \R^m$ be a bounded, connected domain and let $\Omega_j\nearrow \Omega$ an exhaustion of $\Omega$. Let $\rho \in L^\infty(\Omega)$ and $\rho_j \in L^\infty(\Omega_j)$ satisfy
    \begin{align}\label{eq:conv_max_hyp}
        \int_{\Omega_j} \rho_j \, dx = 0, \qquad
        \sup_j \, \norm{\rho_j}_{L^\infty(\Omega_j)} < \infty, \qquad
        \1_{\Omega_j}\rho_j \rightharpoonup \rho
        \quad \text{weakly in $L^1(\Omega)$.}
    \end{align}
    For each $j$, let $u_j$ be the maximizer of $I_{\rho_j,0}$ in $\mathring A(\Omega_j)$, and assume that $u_j \to u$ locally uniformly in $\Omega$ for some $u : \Omega \to \R$. Then $u \in \mathring A(\Omega)$ and $u$ is the maximizer of $I_{\rho,0}$ in $\mathring A(\Omega)$.
\end{lem}
\begin{proof}
    It is easily verified that $u \in \mathring A(\Omega)$. Let $v \in \mathring A(\Omega)$ be any competitor and set
    \begin{align*}
        v_j \doteq v|_{\Omega_j} - \fint_{\Omega_j} v \, dx
        \ \in\ \mathring A(\Omega_j).
    \end{align*}
    By the maximality of $u_j$ and the zero-mean condition $\int_{\Omega_j}\rho_j = 0$,
    \begin{align*}
        I_{\rho_j,0}(u_j)
        \ \geq\ I_{\rho_j,0}(v_j)
        \ =\ \int_{\Omega_j} \sqrt{1 - |Dv|^2} \, dx
        + \int_{\Omega_j} v \rho_j \, dx.
    \end{align*}
    Using the upper semicontinuity of the functional and \eqref{eq:conv_max_hyp} yields $I_{\rho,0}(u) \geq I_{\rho,0}(v)$.
\end{proof}

\subsection{No-light-segment theorems}

Recall that $u \in A(\Omega)$ is said to have a \emph{light segment}
% on $\overline{xy}\doteq \set{sx + (1-s)y \ | \ s\in[0,1]}$
if $\overline{xy}\subseteq\overline\Omega$ and
\begin{align*}
    u((1 - s)x + sy) = u(x) + s|y - x|, \qquad \forall s \in [0,1]
\end{align*}
or, equivalently, $u(y) - u(x) = |y - x|$. Notice that according to our definition the roles of $x$ and $y$ are not symmetric. We call $\overline{xy}$ a \emph{maximal light segment} if it is not included in any other light segment.
% More precisely $\overline{xy}$ is a maximal light segment if it is a light segment and the following hold true
% \begin{align*}
%     u((1 - s)x + sy) > u(x) + s|y - x| &\qquad \forall s \in \set{t < 0 \ | \ tx + (1 - t)y \in\overline{\Omega}} \\
%     u((1 - s)x + sy) < u(x) + s|y - x| &\qquad \forall s \in \set{t > 1 \ | \ tx + (1 - t)y \in \overline{\Omega}}.
% \end{align*}

\begin{rmk}\label{rmk:diff_light_segment}
    If $\overline{xy}$ is a light segment for $u \in A(\Omega)$, then $u$ is 
    differentiable on $\overline{xy} \setminus \set{x,y}$ with constant gradient
    \begin{align}\label{eq:grad_light_segment}
        Du = \frac{y - x}{|y - x|}.
    \end{align}
    In particular, two light segments can only be nested or meet at their endpoints, and every light segment is contained in one and only one maximal light segment.
\end{rmk}

In this Section we will prove that, if $\rho \in L^\infty(\Omega)$, then every maximal light segment of the solution to \cref{prb:variational_problem} must have endpoints on the boundary (\cref{thm:antipeeling} below) and that, if $\psi \in L^\infty(\del\Omega)$, then it cannot have any internal light segment.
Heuristically these results say that a bounded mean curvature $\rho$ does not have enough energy to generate a light segment by itself: if there is one it must be induced ``artificially'' by a Dirichlet boundary datum, while, as a matter of fact, no bounded capillary datum can produce one.

% That a light segment cannot be created under bounded curvature should sound sensible, at least under the strong assumption that $u$ is smooth and every principal curvature is bounded...

Both Bartnik--Simon's and our \cref{thm:capillary_no_light} are proven by a comparison argument with the constant mean curvature hypersurfaces
we are about to describe. For a fixed $o \in \R^m$ we seek functions that are radially symmetric around $o$, that is, of the form $v(x) = \upsilon(|x - o|)$ for some $C^2$-regular function $\upsilon: (0,\infty)\to \R$ whose mean curvature is a chosen constant $\Lambda\in\R$. It is easily seen that the mean curvature equation in polar coordinates around $o$ reads
\begin{align*}
    \frac{1}{\rho^{m-1}}\frac{d}{d\rho}\left(\frac{\rho^{m-1}\upsilon'(\rho)}{\sqrt{1 - |\upsilon'(\rho)|^2}}\right) = \Lambda
\end{align*}
where $\upsilon' = \tfrac{d\upsilon}{d\rho}$. Integrating this equation gives a family of constant mean curvature functions (CMC, for short)
\begin{align}\label{eq:barriers}
    v_{o,\Lambda,K}^{\pm}(x) = v_{o,\Lambda,K}^{\pm}(o) \, \pm \int_0^{|x-o|} \frac{K + \frac{\Lambda s^m}{m}}{\sqrt{s^{2(m-1)}+\left( K + \frac{\Lambda s^m}{m} \right)^2}} \, ds,
\end{align}
where $K$ is an integration constant.

\begin{rmk}\label{rmk:prop_v}
    The functions $v_{o,\Lambda,K}^\pm$ enjoy the following properties:
    \begin{enumerate}[label=(\roman*)]
    \item\label{item:v_space-like} $v_{o,\Lambda,K}^\pm$ is strictly space-like in $\R^m\setminus\{o\}$;
    \item\label{item:vCMC} $v_{o,\Lambda,K}^\pm$ is a classical solution to $H_{v^\pm} = \pm \Lambda$ in $\R^m\setminus\{o\}$; more generally, it is a weak solution with $H_{v^\pm} = \pm\Lambda + \omega_{m-1}K\delta_o$ in $\R^m$;
    \item\label{item:cone_singularity} for each $K>0$ the function $v_{o,\Lambda,K}^\pm$ has a light cone singularity at $o$, namely
    \begin{align*}
        \lim_{|x - o|\to 0}\frac{v_{o,\Lambda,K}^\pm(x)- v_{o,\Lambda,K}^\pm(o)}{|x - o|} = \pm 1,
    \end{align*}
    while for $K=0$ it is a hyperboloid;
    \item\label{item:v_converge} for each $\Lambda\in\R$ the functions $v_{o,\Lambda,K}^\pm$ converge to $v_\infty^{\pm}$ as $K\to \infty$ uniformly on compact sets of $\R^m$ where
    \begin{align*}
        v_{o,\infty}^\pm (x) = v_{o,\Lambda,K}^{\pm}(o) \pm |x - o|
    \end{align*}
    is the \emph{future} (resp. \emph{past}) \emph{light cone} centred at $(v_{o,\Lambda,K}^{\pm}(o),o)$.
\end{enumerate}
\end{rmk}

We recall the \emph{anti-peeling Theorem} of Bartnik and Simon and observe that, since its proof only relies on studying what happens in the interior of the domain, we can state it in a slightly more general form that allows for mixed boundary data.

\begin{thm}[{\cite[Theorem 3.2]{BartnikSimon1982}}]\label{thm:antipeeling}
    Let $u$ be the solution to \cref{prb:variational_problem}. If $\rho \in L^\infty(\Omega)$, then every maximal light segment for $u$ has both ends on the boundary of $\Omega$.
\end{thm}

\begin{rmk}\label{rmk:consequence_antipeeling}
    As an immediate consequence of \cref{thm:antipeeling} we have that, if $\Gamma = \del\Omega$ and $\varphi = \bar\varphi|_{\del\Omega}$ for some $\bar\varphi \in S(\Omega)$, then the maximizer of $I_{\rho}$ in $A_\varphi(\Omega)$ cannot have a light segment. Indeed, if there was one, then the maximal light segment $\overline{xy}$ that it generates must reach the boundary, that is $x,y \in\del\Omega$ and $\varphi(y) - \varphi(x) = |x - y|$, against the assumption on $\varphi$.
\end{rmk}

We can now prove the no-light-segment \cref{thm:capillary_no_light} that shows that the conclusion of \cref{rmk:consequence_antipeeling} holds even when $\Gamma \neq\del\Omega$, provided the capillary datum $\psi$ is bounded and $\Omega$ is convex. The Theorem does not exclude the possibility that the maximizer has light segments on the boundary. As we shall see later, this eventuality will not concern us.

\begin{thm}\label{thm:stronger_no_light}
    If $\Omega$ is convex, $\rho \in L^\infty(\Omega)$ and $\psi \in L^\infty(\del\Omega)$, the solution to \cref{prb:variational_problem} does not have interior light segments.
\end{thm}
\begin{proof}
Assume by contradiction that there is a maximal light segment $\overline{xy}\subset\Omega$. By \cref{thm:antipeeling} it must be $x,y \in \del\Omega$ and since $\varphi$ satisfies \eqref{eq:hp_varphi}, we can assume that $x\in \Nu$. We will construct a comparison CMC of the form \eqref{eq:barriers} centred at some point $o$ on the light segment and apply the comparison principle on $\Omega' = B_r(x) \cap \Omega$, where $r = |x - o|$. The key point is to show that we can choose the parameter $K$ in \eqref{eq:barriers} big enough independently of $r$ in such a way that the premises in \eqref{eq:comparison} are fulfilled.

\textit{Step 1: setup for the comparison argument.}
Set $p\doteq\tfrac{y-x}{|y-x|}$ and fix $0<r<|x-y|$ small enough so that $B_r(x)\cap\del\Omega\subseteq \Nu$, which is possible since $\Nu$ is open in $\del\Omega$. Let $o=o_r$ be the intersection of $\overline{xy}$ with $\del B_r(x)$, i.e.\ $o=x+rp$; since $r<|x-y|$, $o$ is an interior point of $\overline{xy}$, so $o\in\Omega$. Define
\begin{align*}
    \Omega' \doteq \Omega\cap B_r(x), \qquad \Gamma' \doteq \overline{\Omega}\cap\del B_r(x), \qquad \Nu' \doteq \del\Omega\cap{B_r(x)},
\end{align*}
so that $\Gamma'$ is compact, $\Gamma'\cup\Nu'=\del\Omega'$, $\Omega'$ is a convex open set and $o\in\Nu'$. By a standard fact in Measure Theory (see for instance \cite[Proposition 2.16]{Maggi}), $\Gamma' \subset \Omega$ for a.e. $r >0$. We tacitly restrict to such $r$.

Let $v_o(z)\doteq u(o)-|z-o|$ be the past light cone with apex $(u(o),o)$, and let $\Lambda\doteq-\norm{\rho}_{L^\infty(\Omega)}$. For $K>0$ let $v\doteq v^-_{o,\Lambda,K}$ as in \eqref{eq:barriers}, normalized so that $v(o)=u(o)$; explicitly
\begin{align*}
    v(z) = u(o) - \int_0^{|z-o|} g(s)\,ds, \qquad g(t)\doteq \frac{K+\Lambda s^m/m}{\sqrt{s^{2(m-1)}+\big(K+\Lambda s^m/m\big)^2}}\in(0,1),
\end{align*}
so that $v$ solves $H_v=-\Lambda$ classically on $\R^m\setminus\{o\}$, in particular on $\Omega'$ (recall $o\notin\Omega'$). Define
\begin{align*}
    \psi_v \doteq \frac{Dv\cdot n}{\sqrt{1-|Dv|^2}} \quad \text{on $\Nu$},
\end{align*}
an element of $L^\infty(\Nu)$ depending, as $v$, on $r,\Lambda,K$. We shall prove that for every $r$ we can choose $K$ large enough so that
\begin{align}
    \psi_v &\leq \psi \qquad \text{on $\Nu$}, \label{eq:goal_Nu}\\
    v &\leq u \qquad \text{on $\Gamma$}. \label{eq:goal_Gamma}
\end{align}

\textit{Step 2: comparison on $\Nu'$.}
Since $\del\Omega$ is convex, for $\Haus^{m-1}$-a.e. $z \in \del\Omega$ the tangent plane $T_z\del\Omega$ is well defined and $T_z\del\Omega \cap \Omega =
\varnothing$. Since $o \in \Omega$, it follows that
\begin{align*}
    n(z)\cdot(z - o) = \dist(o, T_z\del\Omega) \geq \dist(o,\del\Omega) > 0
    \qquad \text{for $\mathscr H^{m-1}$-a.e.\ } z \in \del\Omega,
\end{align*}
and therefore, setting $s^*(r) \doteq \sup_{z\in\Nu'} |z - o| \leq 2r$,
\begin{align*}
    \cos\theta(z) \doteq \frac{n(z)\cdot(z-o)}{|z-o|}
    \ \geq\ \frac{\dist(o,\del\Omega)}{s^*(r)} \doteq \cos\theta_*(r)
    \ >\ 0
    \qquad \text{for $\mathscr H^{m-1}$-a.e.\ } z \in \Nu'.
\end{align*}
By radial symmetry, for $z\in\Nu'$ and $s=|z-o|$, $Dv(z)=-g(s)\,\tfrac{z-o}{|z-o|}$, hence
\begin{align*}
    \psi_v(z) = -\cos\theta(z)\,\frac{g(s)}{\sqrt{1-g(s)^2}}, \qquad \text{where} \quad \frac{g(s)}{\sqrt{1-g(s)^2}} = \frac{K}{s^{m-1}}+\frac{\Lambda s}{m}.
\end{align*}
Since $K>0$ and $\Lambda\le0$, the right-hand side is strictly decreasing in $s$, so for $s\leq s^*(r)$
\begin{align*}
    \frac{g(s)}{\sqrt{1-g(s)^2}} \ \geq\ \Phi(K,r)\doteq \frac{K}{s^*(r)^{m-1}} + \frac{\Lambda\,s^*(r)}{m}.
\end{align*}
Once $K\geq \norm{\rho}_{L^\infty(\Omega)}\,s^*(r)^m/m$ (so that $\Phi(K,r)\geq0$), this gives $\psi_v(z)\leq -\cos\theta_*(r)\Phi(K,r)$ for every $z\in\Nu$. Since $\Phi(K,r)\to+\infty$ as $K\to\infty$ ($r$ fixed), there is $K_1=K_1(r)$ such that
\begin{align}\label{eq:K1}
    \psi_v \ \leq\ -\norm{\psi}_{L^\infty(\del\Omega)} \ \leq\ \psi \qquad \text{on $\Nu'$}, \qquad \forall K\geq K_1(r),
\end{align}
which is \eqref{eq:goal_Nu}.

\textit{Step 3: comparison on $\Gamma'$.}
We first record that, $u$ being weakly space-like, $u\geq v_o$ everywhere in $\Omega$; moreover
\begin{align}\label{eq:u_above_cone}
    u>v_o \qquad \text{on $\Gamma'\setminus\{o\}$}.
\end{align}
Indeed, if $u(z)=v_o(z)$ for some $z\in\Gamma'\setminus\{o\}$, the weak space-like inequality forces $\overline{zo}$ to be a light segment for $u$ that is transverse to $\overline{xy}$, contradicting \cref{rmk:diff_light_segment}.

Since $0<r<|x-y|$, $o$ is an interior point of $\overline{xy}$, so by \cref{rmk:diff_light_segment} $u$ is differentiable at $o$ with $Du(o)=p$:
\begin{align*}
    \lim_{z\to o} \frac{u(z)-u(o)-p\cdot(z-o)}{|z-o|} = 0.
\end{align*}
For $z\in\Gamma'$ taking the square of the vector  $z-x=(z-o)+(o-x) = z - o + rp$ gives
\begin{align}\label{eq:sphere_identity}
    p\cdot(z-o) = -\frac{|z-o|^2}{2r}, \qquad \forall z \in \Gamma'.
\end{align}
Combining the two
\begin{align*}
    \frac{u(o)-u(z)}{|z-o|} = \frac{|z-o|}{2r} - \frac{u(z)-u(o)-p\cdot(z-o)}{|z-o|} \ \longrightarrow\ 0 \qquad \text{as $z \to o$ in $\Gamma'$.}
\end{align*}
In particular there is $\delta=\delta(r)\in(0,r)$ such that
\begin{align}\label{eq:near_u}
    u(o)-u(z) \ \leq\ \tfrac12|z-o| \qquad \forall z\in\Gamma'\cap B_\delta(o).
\end{align}
Since $s\mapsto K/s^{m-1}+\Lambda s/m$ is decreasing in $s$, its minimum over $t\in(0,\delta]$ is attained at $s=\delta$; as $K/\delta^{m-1}+\Lambda\delta/m\to+\infty$ when $K\to\infty$ ($r,\delta$ fixed), there is $K_2=K_2(r)$ such that for $K\geq K_2(r)$,
\begin{align*}
    \frac{g(s)}{\sqrt{1-g(s)^2}} = \frac{K}{s^{m-1}}+\frac{\Lambda s}{m} \ \geq\ 1, \qquad \text{thus} \qquad g(t)\geq\tfrac1{\sqrt2}>\tfrac12, \qquad \forall t\in(0,\delta].
\end{align*}
Hence for $z\in\Gamma'\cap B_\delta(o)$, with $t=|z-o|\le\delta$,
\begin{align*}
    v(z)-v(o) = -\int_0^t g(s)\,ds \ \leq\ -\tfrac t2 = -\tfrac12|z-o|,
\end{align*}
which combined with \eqref{eq:near_u} and $v(o)=u(o)$ gives
\begin{align}\label{eq:near_zone}
    v(z) \ \leq\ u(o)-\tfrac12|z-o| \ \leq\ u(z), \qquad \forall z\in\Gamma'\cap B_\delta(o), \ K\geq K_2(r).
\end{align}

The set $\Gamma'\setminus B_\delta(o)$ is compact and, by \eqref{eq:u_above_cone}, $u-v_o$ is continuous and strictly positive on it, so
\begin{align*}
    \mu(r) \doteq \min_{\Gamma'\setminus B_\delta(o)} (u-v_o) \ >0
\end{align*}
(if $\Gamma'\subseteq B_\delta(o)$ there is nothing to prove here). By \ref{item:v_converge} in \cref{rmk:prop_v}, $v=v_{o,\Lambda,K}\to v_o$ uniformly on the compact set $\Gamma'\setminus B_\delta(o)$ as $K\to\infty$ ($\Lambda,r$ fixed), so there is $K_3=K_3(r)$ such that for $K\geq K_3(r)$,
\begin{align}\label{eq:far_zone}
    v \ <\ v_o + \mu(r) \ \leq\ u \qquad \text{on } \Gamma'\setminus B_\delta(o).
\end{align}

\smallskip
Combining \eqref{eq:near_zone} and \eqref{eq:far_zone}, for
\begin{align*}
    K \ \geq\ K_0(r) \doteq \max\big(K_2(r),K_3(r)\big)
\end{align*}
we obtain $v\leq u$ on all of $\Gamma'$, which is \eqref{eq:goal_Gamma}.

\textit{Step 4: comparison and contradiction.}
Fix $r$ as above and $K\geq\max(K_0(r),K_1(r))$, so that both \eqref{eq:goal_Nu} and \eqref{eq:goal_Gamma} hold. \cref{lem:barrier_is_maximizer} in the Appendix shows that $v$ is the maximizer of $I_{\Lambda,\psi_v}$ in $A_{v}(\Omega';\Gamma')$ (since $v$ has a singularity at $o\in \del\Omega'$ we need to pay more attention than in \cref{lem:local_minimizer}). Apply \cref{lem:comparison} on $\Omega'$, with $u_1=u$, $u_2=v$, $\rho_1=\rho$, $\rho_2=\Lambda$, $\psi_1=\psi$, $\psi_2=\psi_v$, and $\varphi_1=u|_{\Gamma'}$, $\varphi_2=v|_{\Gamma'}$: by the choice of $\Lambda = \rho_2$ and \eqref{eq:goal_Nu}, \eqref{eq:goal_Gamma} we get
\begin{align*}
    v-u \ \leq\ \sup_\Gamma(v-u) \ \leq\ 0 \qquad \text{in } \overline{\Omega'}.
\end{align*}
However, since $\overline{xo}$ is a light segment for $u$, this latter inequality implies that $v$ has a light segment in $\overline{xo}$, too, a contradiction.
\end{proof}

\section{Gradient estimates}\label{sec:grad_est}

\subsection{Intrinsic geometry of space-like graphs}

We begin by collecting a few geometric facts about space-like graphs. Given a strictly space-like function $u \in S(\Omega) \cap C^1(\overline\Omega)$, let
\begin{align}\label{eq:Rm_graph}
    \Sigma = F(\overline\Omega), \qquad F(x) = (u(x),x),
\end{align}
and let $g = F^*\scal{\cdot}{\cdot}$, a Riemannian metric since $u$ is space-like. Notice that $(\Omega,g)$ is isometric to $\Sigma$ with the induced metric. In Cartesian coordinates on $\Omega$,
\begin{align*}
    g_{ij} = \delta_{ij} - u_iu_j, \qquad g^{ij} = \delta^{ij} + w^2u^iu^j, \qquad w = \frac{1}{\sqrt{1-|Du|^2}},
\end{align*}
where $u_i \doteq \del_i u$ and $u^i \doteq \delta^{ij}u_j$. Throughout, $\nabla$ denotes the Levi-Civita connection of $g$. In particular the gradient with respect to $g$ of a smooth function $f:\Omega\to \R$ is $\nabla^if = g^{ij}f_j$ where $f_j \doteq \del_j f$. We denote norms as $|\nabla f|^2 = g^{ij}f_if_j$ and $|Df|^2 = \delta^{ij}f_if_j$. Notice that
\begin{align}
    |\nabla u|^2 = w^2|Du|^2 = w^2-1.
\end{align}

The second fundamental form and the mean curvature of $\Sigma$ are
\begin{align}\label{eq:frame_H}
    \II(X,Y) = \langle \babla_{F_*X}N, F_*Y \rangle, \qquad H \doteq \sum_{j=1}^m \scal{\babla_{F_*e_j}N}{F_*e_j},
\end{align}
where $\set e_j$ is a $g$-orthonormal frame on $\Omega$, and can be written in terms of the Euclidean Hessian of $u$ as
\begin{align}\label{eq:RM_second_and_mean}
    \II = wD^2u, \qquad H = \div(wDu),
\end{align}
and, intrinsically, in terms of the graph metric,
\begin{align*}
    w\II = \nabla^2u, \qquad wH = \Delta_gu,
\end{align*}
where $\Delta_g$ denotes the Laplace--Beltrami operator of $(\Omega,g)$. We will also need the following identities for the tilt function $w$. Differentiating directly,
\begin{align}\label{eq:Rm_gradients_w}
    Dw = w^3D^2u(Du), \qquad \nabla w = A\nabla u,
\end{align}
where $A$ denotes, as usual, the shape operator of $F$, i.e.\ $\scal{AX}{Y}=\II(X,Y)$. Taking the $g$-divergence of the second identity yields the Jacobi-type equation
\begin{align}\label{eq:mink_jacobi}
    \Delta_gw = w|A|^2+\scal{\nabla u}{\nabla H}.
\end{align}

Writing $dV_g$ for the volume form of $(\Omega,g)$, computing the determinant of $[g_{ij}]$ gives
\begin{align}\label{eq:Rm_volumes}
    dV_g = w^{-1}dx,
\end{align}
where $dx$ is the Lebesgue measure on $\Omega$.

For every vector field $X$ on $\Mink$ we define
\begin{align}\label{eq:dfn_div_sigma}
    \div_\Sigma X \doteq \sum_{j=1}^m \langle \babla_{F_*e_j}X,F_*e_j \rangle,
\end{align}
so that $H = \div_\Sigma N$. Decomposing $X = F_*X^\top - \scal{X}{N}N$ along $F$ and using \eqref{eq:frame_H} we have
\begin{align}\label{eq:div_sigma}
    \div_\Sigma  X &= \div_\Sigma  X^\top - \scal{X}{N}H.
\end{align}
Integrating \eqref{eq:div_sigma} on a piecewise smooth domain $E\subseteq \Omega$ gives the following integration-by-parts formula
\begin{align}\label{eq:sigma_int_part}
    \int_E \div_\Sigma  X \, \diff V_g = - \int_E H \scal{X}{N} \, \diff V_g + \int_{\del E} \scal{X}{\nu} \, d\Haus_g^{m-1},
\end{align}
where $d\Haus_g^{m-1}$ is the Hausdorff measure of $g$ and $\nu$ is the outer normal of $\del E$ in $(\Omega,g)$.

Finally, we record the integral identities that will serve as our main working tools.
\begin{lem}
    Let $E\subseteq\Omega$ be a compact domain with piecewise smooth boundary. Then for every$f,h~\in~C^2(E)$ we have
    \begin{gather}
        \int_E\big(\scal{\nabla f}{\nabla h}+f\Delta_gh\big)\,dV_g = \int_{\del E}f\scal{\nabla h}{\sigma}\,d\mathscr{H}_g^{m-1}, \label{eq:green_1}\\
        \int_E\big(f\Delta_gh-h\Delta_gf\big)\,dV_g = \int_{\del E}\scal{f\nabla h-h\nabla f}{\sigma}\,d\mathscr{H}_g^{m-1}. \label{eq:green_2}
    \end{gather}
    Moreover, if $h \in C_c^1(E)$, then
    \begin{gather}
    \int_E\big(f\scal{\nabla h}{\nabla u}+h\scal{\nabla f}{\nabla u}\big)\,dV_g = -\int_Ef h\,wH\,dV_g - \int_{\del \Omega} fhw \tanh\gamma \, d\mathscr{H}_g^{m-1}. \label{eq:Rm_weak_graph}
    % \\
    % \int_E\big(f\nabla^i\eta+\eta\nabla^if\big)\,dV_g = -\int_Ef\eta\,N^iH\,dV_g.
    \end{gather}
    where $\gamma$ is the contact angle between the surface and the boundary of the cylinder $\R\times\Omega$ defined in \eqref{eq:contact_angle}.
\end{lem}

Identity \eqref{eq:Rm_weak_graph} is precisely \eqref{eq:Rm_weak_sol}, rewritten in terms of the graph metric. 

\begin{proof}
    The first two identities follow by choosing $X = X^\top = f\nabla h$ in \eqref{eq:sigma_int_part}. Choose instead $X=fh \del_0$, for $h\in C^1_c(\Omega)$ and $f\in C^1(\overline\Omega)$. Noticing that $\del_0^\top = -\nabla u$ and recalling that $w = - \scal{N}{\del_0}$ we have
    \begin{align*}
        -\int_E \left(f\scal{\nabla h}{\nabla u} + h\scal{\nabla f}{\nabla u} \right) \, dV_g = \int_E fhH w \, dV_g - \int_{\del \Omega} fh \scal{\nabla u}{\nu} \, d\mathscr{H}_g^{m-1}
    \end{align*}
    as $\div \del_0 = 0$ and $h \in C_c^1(E)$, hence the only portion of the boundary where $h$ may not vanish is the one lying on $\del \Omega$, where $\sigma = \nu$. We are left with computing $\scal{\nabla u}{\nu} = \nu^iu_i$. Let $n^\flat = n_i \, dx^i$ be the exterior conormal (which is independent on the graph metric), so that
    \begin{align*}
        n^i = e^{ij}n_j, \qquad g^{ij}n_j = n^i + w^2 n^j u_j u^i
    \end{align*}
    Normalizing
    \begin{align}\label{Rm:conormal}
        g^{ij}n_i n_j = 1 + \psi^2, \qquad \nu = \frac{n + w\psi Du}{\sqrt{1 + \psi^2}}, \qquad \psi \doteq \scal{N}{n} = - \sinh\gamma
    \end{align}
    and therefore
    \begin{align*}
        \langle \nabla u, \nu \rangle = n^iu_i\frac{1 + w^2|Du|^2}{\sqrt{1 + w^2 (Du\cdot n)^2}} = \frac{w^2 Du\cdot n}{\sqrt{1 + w^2(Du\cdot n)^2}} = w\frac{\psi}{\sqrt{1 + \psi^2}}.
    \end{align*}
    Since $\sqrt{1 + \psi^2} = \cosh\gamma$ we conclude
    \begin{align*}
        \langle \nabla u, \nu \rangle =  - w\tanh\gamma
    \end{align*}
    and we are done.
\end{proof}

The Lorentzian distance on an achronal graph $\Sigma$ is defined by
    \begin{align}\label{eq:BS_ell}
        \ell(x,y) \doteq \sqrt{\abs{x-y}^2 - \big(u(x)-u(y)\big)^2}, \qquad x,y \in \overline{\Omega}
    \end{align}
    writing simply $\ell(x)=\ell(x,o)$ when the centre $o$ is fixed. Let $P$ be the position vector from the origin $(u(o),o)$. 
    % We consider it to be a vector field along the graph letting
    % \begin{align*}
    %     X(x) = (u(x) - u(o),x - o) \in T_{F(x)}M, \qquad x \in \Omega.
    % \end{align*}
    A direct computation gives the two identities
    \begin{align}\label{eq:BS_dist_identities}
        \abs{\nabla\ell}^2 = 1 + \ell^{-2}\scal{N}{P}^2, \qquad \Delta_g\big(\tfrac12\ell^2\big) = m + H\,\scal{N}{P}.
    \end{align}
    We denote by
    \begin{align*}
        L_R(o)=\set{x\in \Omega \ | \ \ell(x,o)<R}
    \end{align*}
    the \emph{Lorentzian ball}. Notice that, if the graph is only weakly space-like, then the set $\set{L_\rho(o)}_{\rho>0}$ is not a fundamental system of neighbourhoods of $o$ in general, as the graph could have a light segment. 

\subsection{The monotonicity formula}\label{subsec:monotonicity} First we review an integral inequality due to Bartnik and Simon involving the volume density $v \doteq w^{-1} = \sqrt{1 - |Du|^2}$ of $dV_g$ with respect to the Lebesgue measure $dx$. 

\begin{thm}[{\cite[Lemma 2.1]{BartnikSimon1982}}]\label{thm:BS_monotonicity}
    Let $\Omega\subseteq\R^m$ be a domain and let $u\in C^2(\Omega)$ be strictly space-like with bounded mean curvature, $\norm{H}_{L^\infty(\Omega)}\leq\Lambda$. Let $o\in\Omega$ and $R>0$ be such that $L_{2R}(o)\Subset\Omega$. Then there exist constants $\alpha<\tfrac1m$ and $c>0$, depending only on $m$, such that
    \begin{align}\label{eq:BS_2_22}
        c\,\exp\!\big(4(\Lambda^2R^2 + 1)\big)\,v(o)^{{\alpha}{}}
        \geq \frac{1}{R^m}\!\!\int_{L_R(o)}\! v^{{\alpha+1}{}}\,\diff x
          + \frac{1}{R^{m-2}}\!\!\int_{L_R(o)}\! |D^2u|^2\diff x.
    \end{align}
\end{thm}

The absence of light segments is paramount for \cref{thm:BS_monotonicity} to be applicable, as it is equivalent to the existence, for any $o\in\Omega$, of a compactly contained Lorentzian ball $L_R(o)\Subset\Omega$. Let us be a little sloppy and forget for a moment that \cref{thm:BS_monotonicity} only applies to \emph{strictly} space-like to illustrate its heuristic content. It says that, for a space-like $u$ with bounded mean curvature and no light segments, we have
\begin{align}\label{eq:implicazione_sbagliata}
    \text{$u$ is light-like at a point} \quad \Longrightarrow \quad \text{$u$ is lightlike in a neighbourhood of the point}.
\end{align}
If the solution $u$ of the variational \cref{prb:variational_problem} problem does not have light segments, then one can apply \cref{thm:BS_monotonicity} to a sequence of strictly space-like functions approximating $u$ to conclude that the implication \eqref{eq:implicazione_sbagliata} holds for it (cf. Step 4 in \cref{subsect:Rm_proof} and the proof of \cite[Theorem 4.1]{BartnikSimon1982}). This would suffice to conclude that it is uniformly space-like, for if it goes null at a point it would go null on a neighbourhood of the point, contradicting the absence of light segments.

Our contribution is noticing that, for space-like graphs with free boundary, the monotonicity formula holds at any $o \in \overline\Omega$, without requiring the Lorentzian ball to be compactly contained in $\Omega$.

\begin{lem}\label{lem:BS_monotonicity}
    Let $\Omega\subseteq\R^m$ be a $C^2$, convex domain and let $u\in C^2(\overline\Omega)$ be strictly space-like and free boundary (i.e. $Du \cdot n \equiv 0$) with bounded mean curvature, $\norm{H_u}_{L^\infty(\Omega)}\leq\Lambda$. Let $o\in\overline\Omega$ and for any $R>0$ let $E_R \doteq L_R(o) \cap \Omega$. Then there exist constants $\alpha<\tfrac1m$ and $c>0$, depending only on $m$, such that
    \begin{align}\label{eq:mia_monotonia}
        c\,\exp\!\big(4(\Lambda^2R^2 + 1)\big)\,v(o)^{{\alpha}{}}
        \geq \frac{1}{R^m}\!\!\int_{E_R}\! v^{{\alpha+1}{}}\,\diff x
          + \frac{1}{R^{m-2}}\!\!\int_{E_R}\! |D^2u|^2\diff x.
    \end{align}
\end{lem}

The proof is a modification of that of \cref{thm:BS_monotonicity}, we include full details for completeness.

\begin{proof}[Proof of \cref{lem:BS_monotonicity}]
    Throughout, $c$ denotes a positive constant depending only on $m$, possibly changing from line to line. Without loss of generality $o$ is the zero vector and $u(o)=0$. We will prove \eqref{eq:mia_monotonia} assuming $u\in C^3(\Omega)$, the general case follows by approximation. Call simply $L_\rho = L_\rho(o)$ and $\del L_\rho = \Omega \cap \del L_\rho(o)$.
    
    \smallskip\noindent
    \emph{Step 1: integrate by parts and discard boundary terms.}
    Let $0\leq f\in C^2(\overline{\Omega})$ and let $\sigma = \abs{\nabla\ell}^{-1}\nabla\ell$ be the outer unit conormal of $\del L_\rho$ in $\Sigma$. Inserting $f$ and $h=\tfrac12(\rho^2-\ell^2)$ into Green's identity \eqref{eq:green_2} and
    using~\eqref{eq:BS_dist_identities}, one obtains
    \begin{align}\label{eq:step_one_monotonicity}
        \int_{E_\rho}\!\Big(m f + \tfrac12(\rho^2-\ell^2)\Delta_g f
            + fH\scal{N}{P}\Big)\diff V_g
        &= \int_{\del L_\rho}\! f\,\ell\,\abs{\nabla\ell}\, d\Haus_g^{m-1} \\ &\quad + \int_{L_\rho \cap \del \Omega} \scal{f\ell\nabla \ell +\tfrac{1}{2}(\rho^2 - \ell^2)\nabla f}{\nu} \, d\Haus_g^{m-1}.\nonumber
    \end{align}
    
    We show that the second contribution in the right-hand side is non-negative and can be discarded. Call $X = x - o$ the position vector field on $\R^m$ centred at $o$. Notice that free boundary condition implies $Du\cdot n = 0$ and $\ell D\ell \cdot n = X \cdot n$ along $\del\Omega$. Recalling \eqref{Rm:conormal} we have
    \begin{align*}
        \scal{\nabla\ell}{\nu} = g^{ij}\ell_in_j = \delta^{ij}\ell_i n_j + w^2 u^i u^j \ell_i n_j = D\ell\cdot n = \ell^{-1} X\cdot n \geq 0
    \end{align*}
    where we used that $\Omega$ is convex at the last inequality.

    Choose $f = v^\alpha$ for some $\alpha>0$ small, so that $Df = Dv^\alpha = Dw^{-\alpha} = -\alpha w^{-\alpha - 1} Dw = - \alpha w^{2-\alpha}D^2u(Du)$. Now, taking the derivative of $Du \cdot n \equiv 0$ along $Du$ gives
    \begin{align}\label{eq:derive_free_boundary}
        0 = D_{Du}Du \cdot n + Du \cdot D_{Du}n = D^2u(Du,n) + \II_{\del\Omega}(Du,Du),
    \end{align}
    where $\II_{\del\Omega}$ is the second fundamental form of $\del\Omega$, hence
    \begin{align*}
        \scal{\nabla v^\alpha}{\nu} = Dv^\alpha\cdot n = - \alpha w^{2 - \alpha}D^2u(Du,n) = \alpha w^{2-\alpha}\II_{\del\Omega}(Du,Du) \geq 0,
    \end{align*}
    again, because the domain is convex. Hence \eqref{eq:step_one_monotonicity} becomes
    \begin{align*}
        \int_{E_\rho}\!\Big(m v^\alpha + \tfrac12(\rho^2-\ell^2)\Delta_g v^\alpha
            + v^\alpha H\scal{N}{P}\Big)\diff V_g
        \geq \int_{\del L_\rho}\! v^\alpha\,\ell\,\abs{\nabla\ell}\,\diff \Haus_g^{m-1}
    \end{align*}
    with equality in case $L_\rho \Subset \Omega$.
    
    Combining this with the coarea formula
    \begin{align*}
        \frac{\diff}{\diff\rho}\left(\int_{E_\rho}f\,\diff V_g\right) =\int_{\del L_\rho} f\,\abs{\nabla\ell}^{-1}\,\diff \Haus_g^{m-1}
    \end{align*}
    and differentiating $\tfrac{1}{\rho^m}\int_{E_\rho}f\,\diff V_g$ yields the
    {monotonicity formula}
    \begin{align}\label{eq:BS_monotonicity}
        \frac{\diff}{\diff\rho}\Big( \frac{1}{\rho^m}\!\!\int_{E_\rho}\!v^\alpha\,\diff V_g\Big)
        &\leq \frac{1}{\rho^{m+1}}\!\!\int_{E_\rho}\!\Big(\tfrac12(\rho^2-\ell^2)\Delta_g v^\alpha
            + v^\alpha H\scal{N}{P}\Big)\, \diff V_g \nonumber
          \\ &\quad - \frac{\diff}{\diff\rho}\Big(\!\int_{E_\rho}\! v^\alpha\,\frac{\scal{N}{P}^2}{\ell^{m+2}}\diff V_g\Big).
    \end{align}
    Since $u$ is $C^2$ and strictly space-like, $\scal{N}{P}=O(\abs{x}^2)$
    as $x\to 0$, so the last term tends to $0$ as $\rho\downarrow0$; and
    if $f$ is continuous at $o$ then
    $\rho^{-m}\int_{E_\rho}f\,\diff V_g\to\omega_m f(o)$, with
    $\omega_m$ the volume of the unit ball in $\R^m$, if $o \in \Omega$; the limit halves if $o \in \del\Omega$. From now on the proof is identical to that given by Bartnik and Simon.
    
    \smallskip
    \noindent\emph{Step 2: the Laplacian of $f=v^\alpha$.} By the Jacobi equation \eqref{eq:mink_jacobi} one readily deduces
    \begin{align}\label{eq:BS_jacobi_v}
        \Delta_g v &= -v\,\abs{A}^2 + \frac{2}{v}\,\abs{\nabla v}^2
            - v^2\,\scal{\nabla u}{\nabla H}
    \end{align}
    and, for $\alpha>0$,
    \begin{align*}
        \Delta_g v^\alpha
        &= \alpha v^{\alpha-1}\Delta_g v
            + \alpha(\alpha-1)v^{\alpha-2}\abs{\nabla v}^2 \nonumber\\
        &= -\alpha v^\alpha\abs{A}^2
            + \alpha(\alpha+1)v^{\alpha-2}\abs{\nabla v}^2
            - \alpha v^{\alpha+1}\scal{\nabla u}{\nabla H}.
    \end{align*}
    i.e. 
    \begin{align}\label{eq:BS_lapl_vgamma}
        \Delta_g v^\alpha = \alpha v^\alpha\Big[-|A|^2+(\alpha+1)\Big(v^{-2}|\nabla v|^2+H\langle\nabla u,\nabla v\rangle\Big)\Big]-\alpha\langle\nabla u,\nabla(Hv^{\alpha+1})\rangle
    \end{align}
    By \eqref{eq:RM_second_and_mean} $\II = v^{-1}D^2u$ hence
    \begin{align*}
        |A|^2 = v^{-2}|D^2u|^2 + 2v^{-4} |D^2u(Du,\cdot)|^2 + v^{-6}D^2u(Du,Du)^2.
    \end{align*}
    On the other hand since $\nabla v = -v^2A\nabla u$ we have
    \begin{align*}
        \scal{\nabla u}{\nabla v} = - v^{-3}D^2u(Du,Du), \qquad v^{-2}|\nabla v|^2 = v^{-4}|D^2u(Du)|^2 + v^{-6}D^2u(Du,Du)^2
    \end{align*}
    thus the term $B$ in brackets in \eqref{eq:BS_lapl_vgamma} is
    \begin{align*}
        B  = -v^{-2}|D^2u|^2+(\alpha-1)v^{-4}|D^2u(Du,\cdot)|^2+\alpha v^{-6}D^2u(Du,Du)^2-(\alpha+1)v^{-3}H\,D^2u(Du,Du).
    \end{align*}
    Since $\alpha < 1$ the second term can be discarded while recalling the mean curvature equation $H = \div(wDu) = w\Delta u + w^3D^2u(Du,Du)$ we can write this as
    \begin{align*}
        B &\leq -v^{-2}|D^2u|^2 + \alpha v^{-6}D^2u(Du,Du)^2-(\alpha+1)v^{-3}H D^2u(Du,Du) \\
        & = - v^{-2}|D^2u|^2 - H^2 + (1 - \alpha)Hv^{-1}\Delta u + \alpha v^{-2}(\Delta u)^2 \\
        &\leq (m(\alpha + \eps) - 1)v^{-2}|D^2u|^2 + \left(\tfrac{(1 - \alpha)^2}{4\eps} - 1 \right) H^2
    \end{align*}
    where we used $(\Delta u)^2 \leq m|D^2u|^2$ and Young's inequality. Thus, there exist constants $c$ and $\alpha<\tfrac{1}{m}$ depending only on $m$ such that
    \begin{align*}
        \Delta_g v^\alpha \leq - c v^{\alpha-2}|D^2u|^2 + \tfrac{1}{4} v^\alpha H^2 - \alpha \scal{\nabla u}{\nabla(H v^{\alpha + 1})}.
    \end{align*}
    % Impose $\eps = \tfrac{(1 - \alpha)^2}{5}$, then find $0<\alpha<1/m$ small enough so that $\eps$ exists and is positive.
    Plugging this in the monotonicity formula  \eqref{eq:BS_monotonicity} shows that the function
    \begin{align*}
        \Psi(\rho)\doteq \tfrac{1}{\rho^m}\int_{E_\rho} v^\alpha \, dV_g
    \end{align*}
    satisfies
    \begin{align}\nonumber
        \Psi'(\rho)
        &\leq - \frac{c}{\rho^{m+1}}\!\!\int_{E_\rho}\tfrac12(\rho^2-\ell^2)\,v^{\alpha-2}|D^2u|^2 \, dV_g \, + \\ &\quad +\frac{1}{\rho^{m+1}}\int_{E_\rho}\!\tfrac12(\rho^2 - \ell^2)\Big(\tfrac{1}{4}\,v^\alpha H^2-\alpha\scal{\nabla u}{\nabla(Hv^{\alpha+1})} \Big) \, dV_g \nonumber \\
        &\quad + \frac{1}{\rho^{m+1}} \int_{E_\rho} v^\alpha H\scal{N}{P}\diff V_g - \frac{\diff}{\diff\rho}\Big(\!\int_{E_\rho}\! v^\alpha\,\frac{\scal{N}{P}^2}{\ell^{m+2}}\diff V_g\Big).\label{eq:polifemo}
    \end{align}

    \smallskip\noindent
    \textit{Step 3: the terms involving $H$.} Apply \eqref{eq:Rm_weak_graph} with $h = \tfrac{1}{2}(\rho^2 - \ell^2)$ and $f = Hv^{\alpha + 1}$:
    \begin{align}\label{eq:anastasia}
        -\int_{E_\rho}\tfrac12(\rho^2-\ell^2)\scal{\nabla u}{\nabla(Hv^{\alpha+1})}\,dV_g
        &= \int_{E_\rho}\tfrac12(\rho^2-\ell^2)v^\alpha H^2\,dV_g \, + \\
        &\quad -\int_{E_\rho}Hv^{\alpha+1}\ell\,\scal{\nabla u}{\nabla\ell}\,dV_g. \nonumber
    \end{align}
    For the cross term, Cauchy--Schwarz and $|\nabla u|^2=w^2-1<v^{-2}$ give $\abs{\scal{\nabla u}{\nabla\ell}}\leq v^{-1}\abs{\nabla\ell}$, so that, for any $R\geq\rho$, Young's inequality $H|\nabla \ell|\leq \tfrac12RH^2 + \tfrac{1}{2R}|\nabla\ell|^2$ and $\rho^2 - \ell^2 \leq \rho^2$ yield
    \begin{align*}
        -\frac\alpha2\int_{E_\rho}(\rho^2-\ell^2)\scal{\nabla u}{\nabla(Hv^{\alpha+1})}\,dV_g
        &\leq \frac\alpha2 \int_{E_\rho} \rho v^\alpha \left (\tfrac{1}{R}|\nabla\ell|^2 + (\rho + R)H^2 \right) \, dV_g \\
        &= \frac\alpha2 \int_{E_\rho} \rho v^\alpha \left (\tfrac{1}{R} + \tfrac1R\tfrac{\scal NP^2}{\ell^2} + (\rho + R)H^2 \right) \, dV_g \\
        &< \frac{1}{2} \int_{E_\rho} v^\alpha\frac{\scal NP^2}{\ell^2} \, dV_g + \frac{\alpha\rho}{2}\left(\tfrac1R + (\rho + R)\Lambda^2 \right) \int_{E_\rho} v^\alpha \, dV_g \\
        \int_{E_\rho} v^\alpha H \scal{N}{P} \, dV_g &\leq \frac{1}{2} \int_{E_\rho} v^\alpha \frac{\scal{N}{P}^2}{\ell^2} + \frac{1}{2} \int_{E_\rho}v^\alpha \ell^2H^2 \, dV_g \\
        &\leq \frac{1}{2} \int_{E_\rho} v^\alpha \frac{\scal NP^2}{\ell^2} \, dV_g + \frac{\rho^2}{2} \int_{E_\rho}v^\alpha H^2 \, dV_g
    \end{align*}
    Summing up and estimating $|H|\leq \Lambda$, integrals featuring the mean curvature in \eqref{eq:polifemo} are bounded by
    \begin{align*}
        \left( \tfrac18\Lambda^2\rho + \tfrac{\alpha}{2R} + \tfrac12(\rho + R)\alpha\Lambda^2 + \tfrac12 \Lambda^2\rho \right) \Psi(\rho) + \int_{E_\rho} v^\alpha \frac{\scal{N}{P}^2}{\ell^2} \, dV_g.
    \end{align*}
    Since $\alpha<\tfrac1m$ we have $\tfrac18 + \tfrac\alpha2 + \tfrac12 \leq 1$ and inequality \eqref{eq:polifemo} becomes
    \begin{align}
        \Psi'(\rho) - \left(\Lambda^2\rho + \tfrac{\alpha}{2R} + \tfrac R2\alpha\Lambda^2 \right) \Psi(\rho) &\leq
        - \frac{c}{\rho^{m+1}}\!\!\int_{E_\rho}\tfrac12(\rho^2-\ell^2)\,v^{\alpha-2}|D^2u|^2 \, dV_g \, + \nonumber \\ 
        &\quad + \frac{1}{\rho^{m+1}} \int_{E_\rho} v^\alpha\frac{\scal{N}{P}^2}{\ell^2} \, dV_g - \frac{\diff}{\diff\rho}\Big(\!\int_{E_\rho}\! v^\alpha\,\frac{\scal{N}{P}^2}{\ell^{m+2}}\diff V_g\Big). \label{eq:polifemo2}
    \end{align}

    \smallskip\noindent
    \textit{Step 4: Integrate the differential inequality.} Next, we notice that the last two terms are a derivative. By a straightforward application of the coarea formula
    \begin{align}\label{eq:antonia}
        f'(\rho) \int_{E_\rho} h \, dV_g=\frac{\diff}{\diff\rho}\int_{E_\rho}\left( f(\rho) - f(\ell) \right) h \, dV_g.
    \end{align}
    Applying this to $f(\rho)=-\tfrac1m\tfrac{1}{\rho^m}$ and $h = v^\alpha\tfrac{\scal{N}{P}^2}{\ell^2}$ yields
    \begin{align*}
        \frac{1}{\rho^{m+1}} \int_{E_\rho} v^\alpha\frac{\scal{N}{P}^2}{\ell^2} \, dV_g  = \frac{\diff}{\diff\rho} \int_{E_\rho} \left(-\frac{1}{m}\frac{1}{\rho^m} + \frac{1}{m}\frac{1}{\ell^m} \right) v^\alpha \frac{\scal NP^2}{\ell^2} \, dV_g.
    \end{align*}
    Hence letting
    \begin{align}\label{eq:BS_T}
        T(\rho) \doteq \Big(1-\frac1m\Big)\int_{E_\rho}v^\alpha\,\frac{\scal{N}{P}^2}{\ell^{m+2}}\,dV_g + \frac1m\,\rho^{-m}\int_{E_\rho}v^\alpha\,\frac{\scal{N}{P}^2}{\ell^{2}}\,dV_g \;\geq 0,
    \end{align}
    and inequality \eqref{eq:polifemo2} takes the form
    \begin{align}\label{eq:BS_diffineq}
        \Psi'(\rho) - \Big(\Lambda^2\rho+\frac\alpha{2R}+\frac R2\alpha\Lambda^2\Big)\Psi(\rho) \leq -\frac{c}{\rho^{m+1}}\int_{E_\rho}\tfrac12(\rho^2-\ell^2)\,v^{\alpha-2}\abs{D^2u}^2\,dV_g - T'(\rho).
    \end{align}
    Introduce the integrating factor
    \begin{align*}
        I(\rho) = \exp\Big(-\tfrac12\Lambda^2\rho^2-\tfrac\alpha{2R}\rho-\tfrac{\alpha\Lambda^2R}2\rho\Big),
    \end{align*}
    so that \eqref{eq:BS_diffineq} becomes
    \begin{align*}
        (I\psi)'(\rho) &\leq -\frac{c \, I(\rho)}{\rho^{m+1}}\int_{E_\rho}\tfrac12(\rho^2-\ell^2)v^{\alpha-2}\abs{D^2u}^2\,dV_g - I(\rho)T'(\rho) \\ 
        &\leq -\frac{c \, I(R)}{\rho^{m+1}}\int_{E_\rho}\tfrac12(\rho^2-\ell^2)v^{\alpha-2}\abs{D^2u}^2\,dV_g - I(\rho)T'(\rho)
    \end{align*}
    because $I$ is decreasing, $I(\rho)\geq I(R)$ for $\rho\in(0,R)$.
    
    Next we integrate in $\rho$ from $0$ to $R$. First notice that we can discard the last term
    \begin{align*}
        \int_0^R I(\rho)T'(\rho) \, d\rho = I(R)T(R) - I(0)T(0) - \int_0^R I'(\rho) T(\rho) \, d\rho \geq 0.
    \end{align*}
    Regarding the first term, we use Fubini's theorem: for fixed $x\in E_R$ with $\ell=\ell(x)$, the condition $x\in E_\rho$ (i.e.\ $\ell<\rho$) forces $\rho$ to range over $(\ell,R)$, so that
    \begin{align*}
        c\,I(R)\int_0^R\rho^{-m-1}\!\int_{E_\rho}\!\tfrac12(\rho^2-\ell^2)v^{\alpha-2}\abs{D^2u}^2\,dV_g\,d\rho
        = c\,I(R)\int_{E_R}\!S_R(\ell)\,v^{\alpha-2}\abs{D^2u}^2\,dV_g,
    \end{align*}
    where
    \begin{align*}
        S_R(\ell)\doteq\int_\ell^R\!\tfrac12\rho^{-m-1}(\rho^2-\ell^2)\,d\rho.
    \end{align*}
    Finally, since 
    \begin{align*}
        I(\rho)\Psi(\rho) \to c_m v(o)^\alpha, \qquad T(\rho) \to 0, \qquad \rho\to 0,
    \end{align*}
    where $c_m=\omega_m$ if $o \in \Omega$ and $c_m = \tfrac{\omega_m}2$ if $o \in \del\Omega$, integrating the left-hand-side gives $I(R)\psi(R) - c_mv(o)^\alpha$. Summing up
    \begin{align}\label{eq:BS_pre_2_22_bis}
        c_m\,v^\alpha(o) \geq \frac{I(R)}{R^m}\!\int_{E_R}\!v^\alpha\,dV_g + c\,I(R)\int_{E_R}\!S_R(\ell)\,v^{\alpha-2}\abs{D^2u}^2\,dV_g.
    \end{align}

    \smallskip\noindent
    \textit{Step 5: conclusion.} We apply \eqref{eq:BS_pre_2_22_bis} with $2R$ in place of $R$, which is legitimate since the parameter in Young's inequality in Step 3 was only required to exceed $\rho$. A direct computation of the integral defining $S_{2R}$ gives, for $m\geq3$,
    \begin{align}\label{eq:BS_SR}
        S_{2R}(\ell) = \frac1{m(m-2)}\,\ell^{2-m}
            + \frac1{2m}\,\frac{\ell^2}{(2R)^m}
            - \frac1{2(m-2)}\,(2R)^{2-m},
    \end{align}
    while
    \begin{align*}
        S_{2R}(\ell)=\tfrac12\log(2R/\ell)-\tfrac14\big(1-\ell^2/(2R)^2\big)
    \end{align*}
    for $m=2$. In either case $S_{2R}(\ell)>0$ for $\ell<2R$ and, in particular,
    \begin{align*}
        S_{2R}(\ell) \geq c\,R^{2-m} \qquad \text{for } \ell < R.
    \end{align*}
    Discarding in \eqref{eq:BS_pre_2_22_bis} the contribution of $L_{2R}\setminus E_R$ and estimating $(2R)^{-m}\geq c\,R^{-m}$, we obtain
\begin{align*}
    c_m\,v^\alpha(o)
    \geq \frac{c\,I(2R)}{R^m}\int_{E_R}\! v^\alpha\,\diff V_g
        + \frac{c\,I(2R)}{R^{m-2}}\int_{E_R}\! v^{\alpha-2}\abs{D^2u}^2\,\diff V_g.
\end{align*}
We now convert both integrals to Lebesgue measure: since $\diff V_g = v\,\diff x$,
\begin{align*}
    \int_{E_R}\! v^\alpha\,\diff V_g = \int_{E_R}\! v^{\alpha+1}\,\diff x,
    \qquad
    \int_{E_R}\! v^{\alpha-2}\abs{D^2u}^2\,\diff V_g
        = \int_{E_R}\! v^{\alpha-1}\abs{D^2u}^2\,\diff x
        \geq \int_{E_R}\! \abs{D^2u}^2\,\diff x,
\end{align*}
where the last inequality follows from $v\leq1$ and $\alpha<1$. Finally, since
$\alpha<1$,
\begin{align*}
    I(2R)^{-1}
    = \exp\Big(2\Lambda^2R^2 + \alpha + \alpha\Lambda^2R^2\Big)
    \leq \exp\big(4(\Lambda^2R^2+1)\big),
\end{align*}
and \eqref{eq:mia_monotonia} follows.
\end{proof}

\subsection{Gradient estimates for the capillary problem} Here we establish an \emph{a priori} bound on the gradient of a classical solution to the homogeneous capillary problem on a strictly convex domain. Unfortunately, at the moment we are only able to treat the homogeneous condition $\psi = 0$: for general bounded capillary data the boundary step of the argument below breaks down. The technique exploited to obtain such gradient estimate is a Bernstein-type argument and is essentially due to Bartnik (see in particular \cite{Bartnik88}).

\begin{thm}\label{thm:capillary_grad_est}
    Assume $\Omega\subseteq\R^m$ is a bounded, strictly convex domain with $C^2$ boundary. Let $u\in C^2(\overline\Omega)\cap C^3({\Omega})$ be a strictly space-like function whose mean curvature satisfies $\norm{H}_{C^1(\Omega)} \leq \Lambda$ for some $\Lambda\geq 0$ and such that $\psi \doteq wDu \cdot n = 0$ on $\del\Omega$, where $n$ is the exterior unit normal of $\Omega$ and $w = (1 - |Du|^2)^{-1/2}$. Then there exists $\theta = \theta(\Omega,\Lambda,m) \in (0,1)$ such that
    \begin{align*}
        \norm{Du}_{L^\infty(\Omega)} \leq 1 - \theta.
    \end{align*}
    More precisely, there is $\lambda = \lambda(m,\Lambda) > 0$ such that $\norm{w}_{L^\infty(\Omega)} \leq 2e^{\lambda\diam\Omega}$.
\end{thm}

\begin{proof}
    Let $f = we^{\lambda u}$ for some $\lambda\in\R$ to be chosen later and let $x \in \overline\Omega$ be its maximum point. By maximality we have
    \begin{align}\label{eq:pilade}
        w \leq w(x)\,e^{\lambda(u(x) - u)} \qquad \text{in $\Omega$},
    \end{align}
    and since by space-likeness of $u$ we have $u(x) - u \leq \diam\Omega$, it suffices to prove that $w(x) \leq 2$ to conclude. Notice that if $Du(x) = 0$ then $w(x) = 1 \leq 2$, so we will assume $Du(x) \neq 0$ throughout. We will first prove that the maximum $x$ cannot lie on $\del\Omega$ and then that one can choose $\lambda = \lambda(\Lambda,m)$ so that $w(x) \leq 2$ if $x\in\Omega$.

    \textit{Step 1: the maximum cannot occur at $\del\Omega$.} First note that, since $\psi = 0$ and $u$ is strictly space-like, the normal derivative $Du \cdot n$ vanishes identically, hence $Du$ is tangent to $\del\Omega$. Assume by contradiction that $x \in \del\Omega$. Then, since $Du\cdot n = 0$ on $\del\Omega$,
    \begin{align}\label{eq:afaleo}
        0 \leq n(we^{\lambda u})
        = e^{\lambda u} \left( w^3D^2u(Du,n) + \lambda w\, Du \cdot n \right)
        = e^{\lambda u} w^3 D^2u(Du,n)
    \end{align}
    at $x$. Now, taking the derivative of $Du \cdot n \equiv 0$ along the tangent direction $Du$ gives
    \begin{align}\label{eq:derive_neumann_bordo}
        0 = D_{Du}Du \cdot n + Du \cdot D_{Du}n
        = D^2u(Du,n) + \II_{\del\Omega}(Du,Du).
    \end{align}
    Combining this with \eqref{eq:afaleo} we obtain
    \begin{align*}
        0 \leq D^2u(Du,n) = - \II_{\del\Omega} (Du,Du) < 0,
    \end{align*}
    where the last inequality is the strict convexity of $\Omega$ together with $Du(x) \neq 0$: a contradiction.

    \textit{Step 2: interior estimate at the maximum point.} By Step~1, $x$ is an interior point, so $\nabla f = 0$ and $\Delta_g f \leq 0$ at $x$. Expanding and dividing by $e^{\lambda u} > 0$, these read
    \begin{gather}
        \nabla w = -\lambda w\nabla u, \label{eq:castore}\\
        \Delta_g w - \lambda^2 w\abs{\nabla u}^2 + \lambda w\Delta_g u \leq 0.
        \label{eq:polluce}
    \end{gather}
    Recalling the Jacobi equation $\Delta_g w = w|A|^2 + \scal{\nabla u}{\nabla H}$ in \eqref{eq:mink_jacobi} and $\Delta_g u = Hw$, the inequality \eqref{eq:polluce} becomes
    \begin{align}\label{eq:emone}
        w|A|^2 - \lambda^2 w|\nabla u|^2
        \ \leq\ -\scal{\nabla u}{\nabla H} - \lambda H w^2.
    \end{align}
    We need to reabsorb the gradient term $-\lambda^2w|\nabla u|^2$ into the second fundamental form. The key is that, by \eqref{eq:castore} and $\nabla w = A\nabla u$ (see \eqref{eq:Rm_gradients_w}), the vector $\nabla u$ is an eigenvector of $A$ with eigenvalue $\mu \doteq -\lambda w$ at $x$ (we already observed that we can assume $\nabla u \neq 0$ there), and that, since $H = \tr_g A$ is prescribed and bounded, a single eigenvalue cannot blow up without the others compensating.
    
    Precisely, assuming $\nabla u (x) \neq 0$, if $\mu = \mu_1,\mu_2,\dots,\mu_m$ are the eigenvalues of $A$, then $\sum_{j\geq2}\mu_j = H + \lambda w$ and, by the Cauchy--Schwarz inequality,
    \begin{align*}
        |A|^2 = \sum_{j=1}^m \mu_j^2
        \ \geq\ \lambda^2 w^2 + \frac{(H + \lambda w)^2}{m-1}.
    \end{align*}
    Hence, recalling $|\nabla u|^2 = w^2 - 1$,
    \begin{align*}
        w|A|^2 - \lambda^2 w|\nabla u|^2
        \ \geq\ \lambda^2w^3 + \frac{w(H + \lambda w)^2}{m - 1}
        - \lambda^2 w^3 + \lambda^2w
        \ \geq\ \frac{w(H + \lambda w)^2}{m - 1}.
    \end{align*}
    By Young's inequality, $(H + \lambda w)^2 \geq \tfrac{1}{2}\lambda^2 w^2 - H^2$, so that
    \begin{align}\label{eq:oreste}
        w|A|^2 - \lambda^2 w|\nabla u|^2
        \ \geq\ \frac{\lambda^2w^3}{2(m - 1)} - \frac{wH^2}{m-1}.
    \end{align}
    On the other hand, since $|\nabla u| \leq w$ and $|\nabla H| \leq w|DH|$, the right-hand side of \eqref{eq:emone} is bounded by
    \begin{align}\label{eq:elettra}
        -\scal{\nabla u}{\nabla H} - \lambda H w^2
        \ \leq\ w^2|DH| + \lambda|H|w^2
        \ \leq\ (1 + \lambda)\Lambda\, w^2.
    \end{align}
    Plugging \eqref{eq:oreste} and \eqref{eq:elettra} into \eqref{eq:emone}, using $w \geq 1$ and $\tfrac1{m-1}\leq1$, we arrive at
    \begin{align*}
        \frac{\lambda^2 w(x)}{2(m - 1)}
        \ \leq\ (1 + \lambda)\Lambda + \Lambda^2,
    \end{align*}
    that is, $w(x) \leq 2(m-1)\lambda^{-2}\big((1+\lambda)\Lambda + \Lambda^2\big)$. Since the right-hand side is infinitesimal as $\lambda \to \infty$, we can choose $\lambda = \lambda(m,\Lambda)$ so large that $w(x) \leq 2$.

    In all cases, then, \eqref{eq:pilade} holds with $w(x) \leq 2$, and since $u$ is weakly space-like, $u(x) - u \leq \diam\Omega$, whence
    \begin{align*}
        \norm{w}_{L^\infty(\Omega)} \ \leq\ C
        \doteq 2e^{\lambda\diam\Omega}.
    \end{align*}
    Recalling $w = (1-|Du|^2)^{-1/2}$, this reads $\norm{Du}_{L^\infty(\Omega)} \leq 1-\theta$ with $\theta \doteq 1 - \sqrt{1 - C^{-2}} \in (0,1)$, depending only on $m$, $\Lambda$ and $\diam\Omega$.
\end{proof}

% ================================================

\section{Existence theorems}\label{sec:existence}

\subsection{Existence of classical solutions} This paragraph is devoted to proving the following existence result.

\begin{thm}\label{thm:Rm_capillary_existence}
    Let $\Omega \subseteq \R^m$ be a smooth, bounded, strictly convex domain, and let
    \begin{align*}
        \rho \in C^{\infty}(\overline\Omega)
    \end{align*}
    with $\int_\Omega \rho \, dx = 0$. Then the problem \eqref{eq:problem_Rm} admits a strictly space-like classical solution
    $u \in C^\infty(\overline\Omega)$, with
    \begin{align}\label{eq:Rm_existence_bounds}
        \int_\Omega u \, dx = 0, \qquad
        \norm{Du}_{L^\infty(\Omega)} \leq 1 - \theta, \qquad
        \norm{u}_{C^{2,\lambda}(\overline\Omega)} \leq C,
    \end{align}
    where $\lambda\in(0,1)$, $\theta = \theta(m,\Omega,\norm{\rho}_{C^1(\overline\Omega)})$ and
    $C = C(m,\lambda,\Omega,\norm{\rho}_{C^1(\overline\Omega)})$. This solution is unique in the class $\mathring A(\Omega)$ of zero-mean functions and coincides with the maximizer $u_{\rho,0}$ of $I_{\rho,0}$ in
    $\mathring A(\Omega)$ given by \cref{lem:existence_maximizer}.
\end{thm}

The proof is a fixed-point argument. Consider the convex closed subset of $C^{1,\lambda}(\overline\Omega)$
\begin{align*}
    \Q \doteq \Big\{ \, v \in C^{1,\lambda}(\overline{\Omega}) \, \ | \ \,  \norm{v}_{C^{1,\lambda}(\overline\Omega)} \leq K, \quad \norm{Dv}_{L^\infty(\Omega)} \leq 1 - \tfrac{\theta}{2} \, \Big\}
\end{align*}
with $\theta,\lambda\in(0,1)$ and $K >0$ to be chosen later. For every $\eps\in(0,1]$ and $v \in \Q$ consider the linear problem
\begin{align}\label{eq:Rm_frozen_problem}
    \begin{cases}
        a^{ij}(Dv)\,D_{ij}u - \eps u = -\rho &\qquad \text{in $\Omega$} \\[4pt]
        Du \cdot n = 0 &\qquad \text{on $\del\Omega$.}
    \end{cases}
\end{align}
where
\begin{align*}
    a^{ij}(p) = w_p \left(\delta^{ij} + w_p^2 p^ip^j \right), \qquad w_p \doteq \frac{1}{\sqrt{1 - |p|^2}}, \qquad |p| < 1.
\end{align*}
We think of \eqref{eq:Rm_frozen_problem} as a prescribed mean curvature problem where the coefficients of the mean curvature operator have been frozen at $v$. Consider the operator
\begin{align}\label{eq:T_eps}
    T_\eps: \Q \longrightarrow C^{1,\lambda}(\overline\Omega)
\end{align}
that maps every $v \in \Q$ to the unique solution to \eqref{eq:Rm_frozen_problem}. Notice that a fixed point $u$ of $T_\eps$ would be a function whose mean curvature is $\eps u - \rho$.

\begin{lem}\label{lem:Rm_Teps_wellposed}
    Let $\Omega \subseteq \R^m$ a $C^{2,\lambda}$-regular, bounded and convex domain, let $\rho \in C^{0,\lambda}(\overline\Omega)$ and $\eps \in (0,1]$. Then the operator $T_\eps$ in \eqref{eq:T_eps} is a well defined, continuous and compact operator and
    \begin{align}\label{eq:Rm_Teps_bound}
        \sup_{v \in \Q}\ \norm{T_\eps v}_{C^{2,\lambda}(\overline\Omega)}
        \ \leq\ C\left(1 + \frac1\eps\right)\norm{\rho}_{C^{0,\lambda}(\overline\Omega)},
    \end{align}
    with $C = C(m,\lambda,\theta,K,\Omega)$.
\end{lem}
\begin{proof}
    By definition of $\Q$, the coefficients of the linear operator $L_v \doteq a^{ij}(v)\del_{ij}^2 - \eps$ are $C^{0,\lambda}$ and its principal symbol is uniformly elliptic: its eigenvalues are $w_p$ with multiplicity $m-1$ and $w_p^3$ and $|Dv| \leq 1 - \tfrac{\theta}{2}$ implies
    \begin{align*}
        |p|^2 \leq a^{ij}(Dv) p_ip_j \leq \Mu_\theta |p|^2, \qquad \Mu_\theta = \left(\theta\left(1 - \tfrac{\theta}{4} \right) \right)^{-3/2}, \qquad p \in\R^m,
    \end{align*}
    for every $v\in\Q$. Then by \cref{lem:Rm_linear_wellposed}, for every $\eps$ there exists a unique $C^{2,\lambda}$ solution to \eqref{eq:Rm_frozen_problem}, hence $T_\eps$ is well defined. The estimate \eqref{eq:Rm_Teps_bound} follows from the second in \eqref{eq:Rm_linear_estimates}.
    
    The image $T_\eps(\Q)$ is bounded in $C^{2,\lambda}(\overline\Omega)$ and thus precompact in $C^{2,\alpha}(\overline\Omega)$ for every $\alpha<\lambda$, so if $T_\eps$ is continuous, it is also compact. To prove continuity, consider a sequence $v_j \to v$ in $C^{1,\lambda}(\overline\Omega)$ with $v_j, v \in\Q$. Since $T_\eps(\set{v_j})$ is pre-compact in $C^{2,\alpha}$, up to a subsequence $T_\eps v_j\doteq u_j \to u$ in $C^{2,\alpha}(\overline\Omega)$ for some $u$. Since $a(Dv_j) \to a(Dv)$ uniformly on $\overline{\Omega}$, passing to the limit in the equations solved by $u_j$ shows that $u$ is a solution to \eqref{eq:Rm_frozen_problem} with coefficients $a(Dv)$. Since by \cref{lem:Rm_linear_wellposed} the solution is unique, it must be $u = T_\eps v$. Hence, every converging subsequence of $u_j$ has the same limit, so it must be $T_\eps v_j \to T_\eps v$ in $C^{2,\alpha}(\overline{\Omega})$.
\end{proof}

We want to establish the existence of a fixed point. We will use the following criterion, which follows from \cite[Theorem 4.4.3]{LloydDegree}.

\begin{lem}\label{lem:Rm_leray_schauder}
    {Let $B$ be a Banach space and $\Q \subseteq B$
    a bounded, closed convex set with $0 \in \inte \Q$.
    Let $T: \Q \to B$ be a continuous and compact operator.
    If, for every $t \in (0,1]$ every solution $u \in \Q$ of $u = t T u$ belongs to $\inte  \Q$, then $T$
    has a fixed point in $\Q$.}
\end{lem}

The gradient estimate \cref{thm:capillary_grad_est} and Lieberman's  \cref{thm:Rm_lieberman_adapted} in \cref{sub:oblique} provide us with good choices of the parameters that define $\Q$ so that \cref{lem:Rm_leray_schauder} applies.

\begin{lem}\label{lem:fixed_point}
    Assume we take
    \begin{enumerate}
        \item $\theta = \theta(\Omega,\Lambda,m)$ given by \cref{thm:capillary_grad_est} when $\Lambda \doteq 1 + \diam\Omega
         + \norm{\rho}_{C^1(\overline\Omega)}$;
        \item\label{item:Rm_choice_K1} $\lambda \doteq \tfrac\beta2$, where $\beta = \beta(m,\theta) \in (0,1)$ is the H\"older exponent given by \cref{thm:Rm_lieberman_adapted};
        \item $K \doteq 1 + \left(1 + (\diam\Omega)^{\lambda} \right)C_L$ where $C_L = C_L(m,\theta,\Lambda,\Omega)$ is given by \cref{thm:Rm_lieberman_adapted}.
        % e $C_L = C_L(m,\theta,\Lambda,\Omega) > 0$ le costanti della stima conormaledi Lieberman \cref{thm:Rm_lieberman_adapted}: ogni soluzione $u \in C^2(\overline\Omega)$ di un problema in forma di divergenza $\div\big(A(Du)\big) = B$ in $\Omega$, $A(Du)\cdot n = 0$ su $\del\Omega$, con $A(p) = w_p\,p$ uniformemente ellittico su $\set{|p| \leq 1 - \theta/2}$, $\norm{B}_{L^\infty(\Omega)} \leq \Lambda$ e $\norm{u}_{L^\infty(\Omega)} \leq \diam\Omega$, soddisfa
        % \begin{align*}
        %     \norm{u}_{C^{1,\beta}(\overline\Omega)} \leq C_L.
        % \end{align*}
        % \item\label{item:Rm_choice_lambda} Fissiamo infine $\lambda \doteq \min\big(\tfrac\beta2, \lambda_0\big)$, dove $\del\Omega \in C^{2,\lambda_0}$, e
        % \begin{align*}
        %     K_1 \doteq 1 + \big(1 + (\diam\Omega)^{\beta - \lambda}\big)\,C_L,
        % \end{align*}
        % cosicché $\norm{u}_{C^{1,\beta}} \leq C_L$ implichi $\norm{u}_{C^{1,\lambda}} \leq K_1 - 1$.
    \end{enumerate}
    and consider the bounded, closed, convex subset of $C^{1,\lambda}(\overline\Omega)$
    \begin{align*}
        \Q \doteq \Big\{ \, v \in C^{1,\lambda}(\overline{\Omega}) \, \ | \ \,  \norm{v}_{C^{1,\lambda}(\overline\Omega)} \leq K, \quad \norm{Dv}_{L^\infty(\Omega)} \leq 1 - \tfrac{\theta}{2} \, \Big\}.
    \end{align*}
    Then for every $\eps\in(0,1]$ the operator $T_\eps: \Q \to C^{1,\lambda}(\overline\Omega)$ has a fixed point in $\Q$. This fixed point is smooth in $\Omega$ and has zero mean.
\end{lem}
\begin{proof}
    If $u\in\Q$ satisfies $u = tTu$ for $t\in(0,1]$, then it is a solution of 
    \begin{align}\label{eq:eps_fixed_point}
        a^{ij}(Du)u_{ij} = \eps u - t\rho \quad \text{in $\Omega$}, \qquad Du \cdot n = 0 \quad \text{on $\del\Omega$}.
    \end{align}
    Geometrically, the graph of $u$ is a free boundary strictly space-like hypersurface with mean curvature $H_u = \eps u - t\rho$. First notice that $u$ has zero mean:
    \begin{align*}
        0 = \int_{\del\Omega} Du\cdot n \, d\Haus^{m-1} = \int_\Omega \div\left( \tfrac{Du}{\sqrt{1 - |Du|^2}}\right) \, dx = \eps\int_\Omega u \, dx - t\int_\Omega \rho \, dx = \eps \int_\Omega u \, dx.
    \end{align*}
    In particular, since $u$ is space-like and $\Omega$ is convex we have
    \begin{align*}
        \norm{u}_{L^\infty(\Omega)} \leq \diam\Omega.
    \end{align*}

    We want to show that $u \in \inte\Q$. The gradient bound will follow from \cref{thm:capillary_grad_est}. Notice that
    \begin{align*}
        \norm{H_u}_{C^1(\overline\Omega)} \leq \eps \left(\norm{Du}_{L^\infty(\Omega)} + \norm{u}_{L^\infty(\Omega)}\right) + t\norm{\rho}_{C^1(\overline\Omega)} \leq 1 + \diam\Omega +\norm{\rho}_{C^1(\overline\Omega)} \doteq \Lambda.
    \end{align*}
    Moreover, $u \in C^\infty(\Omega)$: since $u \in C^{1,\lambda}(\overline{\Omega})$ the coefficients of \eqref{eq:eps_fixed_point} are $C^{0,\lambda}$ and since $u \in \Q$ the equation is uniformly elliptic, hence the interior Schauder bootstrap iterates thanks to $\rho \in C^\infty(\overline\Omega)$ (cf. \cite[Theorem 6.17]{GilbargTrudinger2001}). Moreover, by construction $u \in C^2(\overline\Omega)$, therefore, by \cref{thm:capillary_grad_est} 
    \begin{align}\label{eq:culo}
        \norm{Du}_{L^\infty(\Omega)} \leq 1 - \theta \leq 1 - \tfrac\theta2
    \end{align}
     by our choice of $\theta$. Finally, by \cref{thm:Rm_lieberman_adapted} we have a H\"older estimate $\norm{u}_{C^{1,\beta}(\overline\Omega)} \leq C_L$, and thus, by our choice of $K$,
    \begin{align}\label{eq:camicia}
        \norm{u}_{C^{1,\lambda}(\overline\Omega)}\leq K - 1.
    \end{align}
    Since \eqref{eq:culo} and \eqref{eq:camicia} hold uniformly in $t$, this shows that every $u$ such that $u = tTu$ lies in $\inte \Q$, hence by \cref{lem:Rm_leray_schauder} $T_\eps$ has a fixed point in $\Q$.
\end{proof}

\begin{proof}[Proof of \cref{thm:Rm_capillary_existence}]
    By \cref{lem:fixed_point} for each $T_\eps$ has a fixed point $u_\eps \in \Q$. By definition this solves $H_{u_\eps} = \eps u_\eps - \rho \doteq F_\eps$. By Schauder estimates \cite[Theorem 6.30]{GilbargTrudinger2001} (or \eqref{eq:Rm_schauder_oblique} below)
    \begin{align*}
        \norm{u_\eps}_{C^{2,\lambda}(\overline\Omega)} \leq
        C(m,\lambda,\theta,K,\Omega)
        \big(\diam\Omega + \norm{F_\eps}_{C^{0,\lambda}}\big) \leq C,
    \end{align*}
    uniformly in $\eps$. By Ascoli-Arzelà, up to a subsequence, $u_\eps \to u$ in $C^{2}(\overline\Omega)$ and the limit is itself $C^{2,\lambda}(\overline\Omega)$. A bootstrap argument based on \cite[Exercise 2.3]{Lieberman2013} shows that $u \in C^\infty(\overline{\Omega})$. Moreover the estimate $\norm{Du_\eps}_{L^\infty(\Omega)} \leq 1 - \theta$ and $\int_\Omega u_\eps = 0$ pass to the limit, as well as the boundary condition $Du_\eps \cdot n = 0$. Now, since $F_\eps \to - \rho$ uniformly as $\eps \to 0$, the limit is a solution to \eqref{eq:problem_Rm}. Finally, by \cref{rmk:solutions_are_max} and the uniqueness in \cref{lem:existence_maximizer}, $u$ is the maximizer of $I_{\rho,0}$ in $\mathring A(\Omega)$
\end{proof}

\subsection{Proof of \cref{thm:Rm_main}}\label{subsect:Rm_proof}

\emph{Step 0: approximating data.} Let $\set{\Omega_j}_j$ be a sequence of smooth, bounded, strictly convex domains such that
\begin{align*}
    \Big\{x \in \Omega \ | \ \dist(x,\del\Omega) > \tfrac1j\Big\} \ \subseteq \
    \Omega_j \ \Subset \ \Omega,
\end{align*}
so that $\Omega_j\nearrow\Omega$. Extend $\rho$ to $\R^m$ by zero, let $\tilde\rho_j \doteq\rho * \phi_{1/j}$ be a standard mollification and set
\begin{align*}
    \rho_j \doteq \tilde\rho_j - \fint_{\Omega_j}\tilde\rho_j \, dx
    \qquad \text{on $\Omega_j$},
\end{align*}
so that $\int_{\Omega_j}\rho_j = 0$, $\rho_j \in C^\infty(\overline\Omega_j)$ and $\norm{\rho_j}_{L^\infty(\Omega_j)} \leq 2\norm{\rho}_{L^\infty(\Omega)} \doteq 2\Lambda$. Moreover, since $\tilde\rho_j \to \rho$ in $L^2(\R^m)$ and $\fint_{\Omega_j}\tilde\rho_j \to \tfrac{1}{|\Omega|}\int_\Omega \rho = 0$ by the zero-mean assumption on $\rho$, we have
\begin{align}\label{eq:rhoj_L2}
    \1_{\Omega_j}\rho_j \longrightarrow \rho \qquad \text{in $L^2(\Omega)$.}
\end{align}
By \cref{thm:Rm_capillary_existence} there exists, for each $j$, a strictly space-like $u_j \in C^{2,\lambda}(\overline\Omega_j)\cap C^\infty(\Omega_j)$ with $\int_{\Omega_j} u_j = 0$ solving
\begin{align}\label{eq:approx_pb_proof}
    - \div\Bigg(\frac{Du_j}{\sqrt{1 - |Du_j|^2}} \Bigg) = \rho_j
    \quad \text{in $\Omega_j$,}
    \qquad Du_j \cdot n_j = 0 \quad \text{on $\del\Omega_j$,}
\end{align}
where $n_j$ is the exterior normal of $\del\Omega_j$; by \cref{rmk:solutions_are_max}, $u_j$ is the maximizer of $I_{\rho_j,0}$ in $\mathring A(\Omega_j)$.

\smallskip\noindent
\emph{Step 1: the limit function.} Each $u_j$ is $1$-Lipschitz on the convex set $\overline\Omega_j$ and has zero mean there, hence $\norm{u_j}_{L^\infty(\Omega_j)} \leq \diam\Omega$. Extending each $u_j$ to $\overline\Omega$ as a $1$-Lipschitz function (\cite[Theorem 3.1]{EvansGariepy2015}), Ascoli--Arzelà yields, up to a subsequence, $u_j \to u$ uniformly on $\overline\Omega$ for some $u$. On every compact subset of $\Omega$, $u$ is a uniform limit of $1$-Lipschitz functions, and $\int_\Omega u = \lim_j \int_{\Omega_j} u_j = 0$ by dominated convergence; hence $u \in \mathring A(\Omega)$. By \cref{lem:converging_of_maximizers}, $u$ is the maximizer of $I_{\rho,0}$ in $\mathring A(\Omega)$.

\smallskip\noindent
\emph{Step 2: consequences of the monotonicity formula.} Set
$v_j \doteq ({1 - |Du_j|^2})^{1/2}$ and $R > \diam\Omega$. For every $o \in \overline\Omega_j$ we have $\Omega_j \subseteq B_R(o) \subseteq L^j_R(o)$, where $L^j_R(o)$ is the Lorentzian ball \eqref{eq:BS_ell} relative to $u_j$, because the Lorentzian distance is dominated by the Euclidean one. Since $u_j$ is a free boundary solution on the convex domain $\Omega_j$ with $\norm{H_{u_j}}_{L^\infty(\Omega_j)} \leq 2\Lambda$, \cref{lem:BS_monotonicity} applies at every $o \in \overline\Omega_j$ and yields a constant $C = C(m,\Lambda,\diam\Omega)$ such that, for every $j$ and every $o \in \overline\Omega_j$,
\begin{align}
    \int_{\Omega_j} |D^2u_j|^2 \, dx \ &\leq\ C\,v_j(o)^\alpha \ \leq\ C,
    \label{eq:mono_D2}\\
    \int_{\Omega_j} v_j^{\alpha+1} \, dx \ &\leq\ C\,v_j(o)^\alpha.
    \label{eq:mono_v}
\end{align}

\smallskip\noindent
\emph{Step 3: local $W^{2,2}$ convergence.} Fix $\Omega' \Subset \Omega$ and let $j$ be so large that $\Omega' \subseteq \Omega_j$. By \eqref{eq:mono_D2} and the uniform $L^\infty$ and Lipschitz bounds, $\set{u_j}$ is bounded in $W^{2,2}(\Omega')$. By a diagonal argument over an exhaustion of $\Omega$, up to a further subsequence
\begin{align*}
    u_j \rightharpoonup u \quad \text{in $W^{2,2}_\loc(\Omega)$,} \qquad
    u_j \to u \quad \text{in $W^{1,2}_\loc(\Omega)$,} \qquad
    Du_j \to Du \quad \text{a.e.\ in $\Omega$,}
\end{align*}
the identification of the limit with $u$ being granted by the uniform convergence of Step 1. In particular $u \in W^{2,2}_\loc(\Omega)$ and, setting $v \doteq \sqrt{1 - |Du|^2}$,
\begin{align*}
    v_j \to v \quad \text{a.e.\ in $\Omega$}, \qquad\text{hence}\qquad
    \int_{\Omega_j} v_j^{\alpha+1}\,dx
    = \int_\Omega \1_{\Omega_j} v_j^{\alpha+1}\,dx
    \ \longrightarrow\ \int_\Omega v^{\alpha+1}\,dx
\end{align*}
by dominated convergence, since $0 \leq v_j \leq 1$ and $\Omega_j \nearrow \Omega$.

\smallskip\noindent
\emph{Step 4: the limit is uniformly space-like.} Suppose first, by contradiction, that $v = 0$ a.e.\ in $\Omega$. Pick a point $o \in \Omega$ with $v_j(o) \to 0$: then, by \eqref{eq:mono_D2} applied at this $o$ and by weak lower semicontinuity of the $W^{2,2}(\Omega')$ seminorm,
\begin{align*}
    \int_{\Omega'} |D^2u|^2\,dx
    \ \leq\ \liminf_{j\to\infty} \int_{\Omega_j} |D^2u_j|^2\,dx
    \ \leq\ C \lim_{j\to\infty} v_j(o)^\alpha = 0
\end{align*}
for every $\Omega' \Subset \Omega$. Hence $Du$ is constant on the connected set $\Omega$ with $|Du| = 1$, that is, the graph of $u$ lies in a lightlike hyperplane; in particular $u$ possesses light segments contained in $\Omega$, contradicting \cref{thm:capillary_no_light}. Therefore
\begin{align*}
    2\delta \doteq \int_\Omega v^{\alpha+1}\,dx > 0,
\end{align*}
and by Step 3 we have $\int_{\Omega_j} v_j^{\alpha+1} \geq \delta$ for all $j \geq j_0$. Feeding this into \eqref{eq:mono_v} gives, for every $j \geq j_0$ and {every} $o \in \overline\Omega_j$,
\begin{align*}
    v_j(o)^\alpha \ \geq\ \frac{\delta}{C},
    \qquad\text{that is}\qquad
    v_j \geq \vartheta \doteq \Big(\frac{\delta}{C}\Big)^{1/\alpha}
    \quad \text{on $\overline\Omega_j$.}
\end{align*}
Equivalently, $w_j \doteq v_j^{-1} \leq \vartheta^{-1}$ and $|Du_j| \leq \sqrt{1 - \vartheta^2} \doteq 1 - \theta$ on $\overline\Omega_j$, uniformly in $j \geq j_0$. Passing to the a.e.\ limit, $|Du| \leq 1 - \theta$ a.e.\ in $\Omega$ and $w \doteq v^{-1} \in L^\infty(\Omega)$.

\smallskip\noindent
\emph{Step 5: passage to the limit in the weak formulation.} Fix $\eta \in
C^1(\overline\Omega)$. Each $u_j$ satisfies \eqref{eq:Rm_weak_sol} on
$\Omega_j$ with $\psi = 0$:
\begin{align*}
    \int_{\Omega} \1_{\Omega_j}\, w_j Du_j \cdot D\eta \, dx
    = \int_{\Omega} \1_{\Omega_j}\, \eta\rho_j \, dx.
\end{align*}
By Step 3 and Step 4, $\1_{\Omega_j} w_j Du_j \to wDu$ a.e.\ in $\Omega$, with $|\1_{\Omega_j} w_j Du_j \cdot D\eta| \leq \vartheta^{-1} \norm{D\eta}_{L^\infty(\Omega)} \in L^1(\Omega)$: dominated convergence gives the convergence of the left-hand side to $\int_\Omega wDu\cdot D\eta$. The right-hand side converges to $\int_\Omega \eta\rho$ by \eqref{eq:rhoj_L2}. Hence $u$ satisfies \eqref{eq:Rm_weak_sol} with $\psi = 0$, and since $w \in L^\infty(\Omega) \subseteq L^1(\Omega)$, $u$ is a weak solution to \eqref{eq:problem_Rm}.

Finally, $u \in \mathring A(\Omega)$ coincides with the maximizer $u_{\rho,0}$ by Step 1, and any weak solution is a maximizer of $I_{\rho,0}$ by \cref{rmk:solutions_are_max}, hence coincides with $u_{\rho,0}$ up to an additive constant by \cref{lem:existence_maximizer}. This proves uniqueness and concludes the proof of \cref{thm:Rm_main}. \qed

\begin{rmk}
    \cref{lem:BS_monotonicity} is used to prove the implication
    \begin{align}\label{eq:giovanna}
        v = 0 \quad \text{a.e.} \quad \Longrightarrow \quad D^2 u = 0 \quad \text{a.e.}
    \end{align}
    which in general is not guaranteed (consider for example a sequence $u_j$ of strongly oscillating, nearly light-like, functions with $|Du| = 1$ a.e. and converging uniformly to $0$), and then again to show that if $\esssup v > 0$ somewhere, then $\essinf v > 0$. \cref{thm:capillary_no_light} intervenes to exclude the conclusion in \eqref{eq:giovanna}, which nonetheless could have been excluded also by other means. For example, arguing as in \cite[Proposition 3.13]{BIMM}, it is not difficult to show that the tilt function of the maximizer satisfies
    \begin{align*}
        \frac{1}{\sqrt{1 - |Du|^2}} \in L^1(\Omega)
    \end{align*}
    at least in case $\Gamma = \varnothing$. The absence of light segments is on the contrary fundamental in \cite{BartnikSimon1982} to ensure the applicability of \cref{thm:BS_monotonicity} and run the argument, for it prevents the Lorentzian balls from touching the boundary. In our case, the free boundary condition is strong enough to ensure the validity of the monotonicity formula regardless of the presence of light segments, the absence of which is proved \emph{a posteriori} by the proof of \cref{thm:Rm_main}.
\end{rmk}

\appendix

\section{Appendix}

\subsection{CMC boundaries are weak solutions} In \cref{item:vCMC} of \cref{rmk:prop_v} it is observed that the mean curvature of a CMC function \eqref{eq:barriers} with vertex at $o$ is of the form $H_v = - \Lambda + \omega_{m-1} K \delta_o$. It is clear that on a domain that avoids $o$, $v$ is a classical solution of a problem \eqref{eq:general_problem} with bounded constant mean curvature. Here we show that if $o$ lies in the Dirichlet part of the boundary of the domain, $v$ has still bounded constant mean curvature in a weak sense. 

We verify that the CMC functions \eqref{eq:barriers} are weak solutions of a problem \eqref{eq:general_problem} with bounded constant mean curvature and capillary derivative, even if their vertex $o$ lies on the Dirichlet component of the boundary. This is not true in the case $o$ is an internal point, for the mean curvature at $o$ gives a Dirac delta contribution, as stated in \cref{item:vCMC} of \cref{rmk:prop_v}.

\begin{lem}\label{lem:barrier_is_maximizer}
    Let $\Omega' \subseteq \R^m$ be a bounded, convex domain, let $\Gamma'\subseteq \del\Omega'$ be a compact set, $\Nu' \doteq \del\Omega'\setminus\Gamma'$, and let $o \in \Gamma'$. Let $\Lambda \leq 0$, $K > 0$ and let $v \doteq v^-_{o,\Lambda,K}$ be the CMC function \eqref{eq:barriers}. Set $\psi_v \doteq w_v Dv \cdot n$ on $\Nu'$, where $n$ is the outward normal of $\del\Omega'$. Then $v$ is the maximizer of $I_{\Lambda,\psi_v}$ in $A_v(\Omega';\Gamma')$.
\end{lem}
\begin{proof}
    Let $\zeta \in A_v(\Omega';\Gamma')$ be any competitor. By the concavity of $f(p) = \sqrt{1-|p|^2}$ on $\overline{B_1}$ the inequality
    \begin{align}\label{eq:barrier_concavity}
        f(D\zeta) \ \leq\ f(Dv) - w_v\, Dv \cdot (D\zeta - Dv)
    \end{align}
    holds a.e.\ in $\Omega'$ (also where $|D\zeta| = 1$, by continuity of $f$). Note that $\eta \doteq \zeta - v \in W^{1,\infty}(\overline{\Omega'})$
    satisfies $\eta = 0$ on $\Gamma'$; in particular, since $v$ is
    $1$-Lipschitz and $\eta(o) = 0$,
    \begin{align}\label{eq:barrier_eta_small}
        |\eta| \leq 2|z - o| \qquad \text{on $\overline{\Omega'}$.}
    \end{align}
    Moreover, by \eqref{eq:barriers} the radial profile of $v$ gives, for
    $s = |z - o|$ small,
    \begin{align}\label{eq:barrier_wv}
        w_v(z) = \frac{\sqrt{s^{2(m-1)} + (K + \Lambda s^m/m)^2}}{s^{m-1}}
        \ \leq\ \frac{c_K}{s^{m-1}},
    \end{align}
    so that $w_v \in L^1(\Omega')$ and $w_v \mathscr H^{m-1}(\del B_\delta(o)) \leq c_K' $ uniformly in $\delta$ small.

    Fix $\delta > 0$ small and set $\Omega'_\delta \doteq \Omega' \setminus \overline{B_\delta(o)}$. Since $v$ is a classical solution on $\overline{\Omega'_\delta}$, integrating by parts,
    \begin{align*}
        -\int_{\Omega'_\delta} w_v\, Dv\cdot D\eta \, dx
        = \int_{\Omega'_\delta} \eta \,\div(w_vDv) \, dx
        - \int_{\del\Omega'_\delta} \eta\, w_v\, Dv\cdot \sigma \,
            d\mathscr H^{m-1},
    \end{align*}
    where $\sigma$ is the outward normal of $\Omega'_\delta$. On $\del\Omega'_\delta$ we distinguish: on $\Gamma' \setminus B_\delta(o)$ the integrand vanishes since $\eta = 0$ there; on $\Nu' \setminus B_\delta(o)$ we have $w_vDv\cdot\sigma = \psi_v$; on $\Omega' \cap \del B_\delta(o)$, by $|Dv \cdot \sigma| \leq 1$, \eqref{eq:barrier_eta_small} and \eqref{eq:barrier_wv},
    \begin{align*}
        \left| \int_{\Omega'\cap\del B_\delta(o)} \eta\, w_v \,Dv\cdot\sigma \,
        d\mathscr H^{m-1} \right|
        \ \leq\ 2\delta \cdot \frac{c_K}{\delta^{m-1}} \cdot
        \omega_{m-1}\delta^{m-1}
        \ =\ c\,\delta \ \longrightarrow\ 0.
    \end{align*}
    Using $\div(w_vDv) = H_v = - \Lambda$ on $\Omega'_\delta$ and letting $\delta \to 0$ (the volume integrals converge by dominated convergence), we obtain
    \begin{align}\label{eq:barrier_ibp}
        -\int_{\Omega'} w_v\, Dv\cdot D\eta \, dx
        = - \int_{\Omega'} \eta\, \Lambda \, dx
        - \int_{\Nu'} \eta\, \psi_v \, d\mathscr H^{m-1}.
    \end{align}
    Integrating \eqref{eq:barrier_concavity} over $\Omega'$ and inserting \eqref{eq:barrier_ibp},
    \begin{align*}
        I_{\Lambda,\psi_v}(\zeta) - I_{\Lambda,\psi_v}(v)
        &= \int_{\Omega'} \big(f(D\zeta) - f(Dv)\big) \, dx
        + \int_{\Omega'} \eta\,\Lambda \, dx
        + \int_{\Nu'} \eta\,\psi_v \, d\mathscr H^{m-1} \\
        &\leq -\int_{\Omega'} w_v\, Dv\cdot D\eta \, dx
        + \int_{\Omega'} \eta\,\Lambda \, dx
        + \int_{\Nu'} \eta\,\psi_v \, d\mathscr H^{m-1}
        \ =\ 0.
    \end{align*}
    Hence $v$ is a maximizer. It is the unique one by \cref{lem:existence_maximizer}.
\end{proof}

\subsection{Oblique derivative boundary conditions}\label{sub:oblique}

Here we include for completeness some well known facts about oblique derivative problems.

\begin{lem}\label{lem:Rm_linear_wellposed}
    Let $\Omega \subseteq \R^m$ a $C^{2,\lambda}$-regular domain, $\lambda \in (0,1)$. Let $a = (a^{ij}) \in C^{0,\lambda}(\overline\Omega; \mathrm{Sym}(m))$ with
    \begin{align*}
        \mu\,|\xi|^2 \leq a^{ij}(x)\,\xi_i\xi_j \leq \mathrm M\,|\xi|^2, \qquad x \in \overline\Omega, \ \xi \in \R^m,
    \end{align*}
    for constants $0 < \mu \leq \mathrm M$, and let $\eps \in (0,1]$, $f \in C^{0,\lambda}(\overline\Omega)$. Then the problem
    \begin{align}\label{eq:Rm_linear_problem}
        \begin{cases}
            a^{ij}\del_{ij}^2u - \eps u = f &\qquad \text{in $\Omega$} \\[4pt]
            Du \cdot n = 0 &\qquad \text{on $\del\Omega$}
        \end{cases}
    \end{align}
    has a unique solution $u \in C^{2,\lambda}(\overline\Omega)$, and the estimates
    \begin{align}\label{eq:Rm_linear_estimates}
        \norm{u}_{L^\infty(\Omega)} \leq \frac{1}{\eps}\,\norm{f}_{L^\infty(\Omega)},
        \qquad
        \norm{u}_{C^{2,\lambda}(\overline\Omega)} \leq C\left(1 + \frac1\eps\right)\norm{f}_{C^{0,\lambda}(\overline\Omega)},
    \end{align}
    hold with $C = C(m, \lambda, \mu, \mathrm M, \norm{a}_{C^{0,\lambda}(\overline\Omega)}, \Omega)$.
\end{lem}
\begin{proof}
     Let $L \doteq a^{ij}\del_{ij}^2 - \eps$. The global Schauder estimates for oblique problems \cite[Theorem 6.30]{GilbargTrudinger2001} give, for every $u \in C^{2,\lambda}(\overline\Omega)$ solution to \eqref{eq:Rm_linear_problem},
    \begin{align}\label{eq:Rm_schauder_oblique}
        \norm{u}_{C^{2,\lambda}(\overline\Omega)}
        \ \leq\ C\Big( \norm{u}_{L^\infty(\Omega)} + \norm{a^{ij}\del_{ij}^2u - \eps u}_{C^{0,\lambda}(\overline\Omega)} \Big),
    \end{align}
    where $C = C(m,\lambda,\mu,\mathrm M,\norm{a}_{C^{0,\lambda}},\Omega)$: the boundary operator $n$ is uniformly oblique and $\norm{n}_{C^{1,\lambda}(\del\Omega)}$ is controlled by the $C^{2,\lambda}$ regularity of $\del\Omega$, and the zero order coefficient satisfies $|-\eps| \leq 1$. 

    To deduce the second in \eqref{eq:Rm_linear_estimates} we apply the maximum principle. Let $u \in C^2(\Omega)\cap C^1(\overline{\Omega})$ be any solution to \eqref{eq:Rm_linear_problem} and let $h \doteq u -  \eps^{-1}\norm{f}_{L^\infty(\Omega)}$, so that $h$ solves
    \begin{align*}
        \begin{cases}
            Lh \geq 0 &\qquad \text{in $\Omega$}, \\ 
            Dh \cdot n = 0 &\qquad \text{on $\del\Omega$.}
        \end{cases}
    \end{align*}
    We want to show that $h \leq 0$. Assume by contradiction that $s\doteq \sup_{\overline\Omega} h > 0$. If the maximum is attained in $\Omega$, then the strong maximum principle \cite[Theorem 3.5]{GilbargTrudinger2001} applies and $h$ must be constant, hence $Lh = -\eps s < 0$ at the maximum point, a contradiction. If the maximum is attained at some point $x\in\del\Omega$ and $h < s$ on the interior, then the Hopf Lemma \cite[Lemma 3.4]{GilbargTrudinger2001} gives $Du \cdot n > 0$ at $x$, a contradiction with the boundary condition. We conclude that $h \leq 0$, that is $\norm{u}_{L^\infty(\Omega)} \leq {\eps}^{-1}{\norm{f}_{L^\infty(\Omega)}}.$ Combining this with \eqref{eq:Rm_schauder_oblique} with $f = a^{ij}\del_{ij}^2u - \eps u$ gives the second in \eqref{eq:Rm_linear_estimates}. Applying the maximum principle to the difference of two solutions also gives uniqueness.

    To show existence, we apply the continuity method. Let
    \begin{align*}
        \mathfrak A \doteq \set{u \in C^{2,\lambda}(\overline\Omega) \ | \ Du \cdot n = 0 \text{ on $\del\Omega$} }, \qquad \mathfrak B \doteq C^{0,\lambda}(\overline\Omega),
    \end{align*}
    and for every $t\in[0,1]$ define
    \begin{align*}
        \mathfrak L_t: \mathfrak A \longrightarrow \mathfrak B, \qquad \mathfrak L_t u \doteq a_t^{ij}\del_{ij}^2u - \eps u, \qquad a_t^{ij} = (1 - t)\delta^{ij} + ta^{ij}.
    \end{align*}
    These are bounded linear operators, moreover they are uniformly elliptic with constants $\min\set{1,\mu}\leq \max\set{1,\Mu}$ and the $C^{0,\lambda}$-norms are uniformly bounded in $t$. By the above discussion the inequality
    \begin{align*}
        \norm{u}_{\mathfrak A} \leq C \left( 1 + \tfrac1\eps \right) \norm{\mathfrak L_tu}_{\mathfrak B}
    \end{align*}
    holds for every $t \in [0,1]$ with the same constant $C$ established above. By \cite[Theorem 5.2]{GilbargTrudinger2001}, $\mathfrak L_1$ is surjective if and only if $\mathfrak L_0 = \Delta - \eps$ is, so we are left to show that the problem
    \begin{align}\label{eq:hoppipolla}
        \Delta u - \eps u = f \quad \text{in $\Omega$} \qquad Du \cdot n = 0 \quad \text{on $\del\Omega$}
    \end{align}
    has a solution in $C^{2,\lambda}(\overline\Omega)$ for every $f \in C^{0,\lambda}(\Omega)$.

    We do so first by applying Lax-Milgram to the bilinear form
    \begin{align*}
        B(u,v) \doteq \int_\Omega Du \cdot Dv \, dx + \eps \int_\Omega uv, \qquad u,v \in W^{1,2}(\Omega)
    \end{align*}
    which is continuous and coercive, hence, there exists a unique weak solution $u \in W^{1,2}(\Omega)$ to \eqref{eq:hoppipolla}. Testing the weak equation with the test $(u - \eps^{-1}\norm{f}_{L^\infty(\Omega)})_+ \in W^{1,2}(\Omega)$ yields $u \leq \eps^{-1}\norm{f}_{L^\infty(\Omega)}$ a.e. on $\Omega$. Then \cite[Theorem 2]{Lieberman88} gives $u \in C^{1,\alpha}$ for some $\alpha\in(0,1)$, in particular it is Lipschitz. Since by \cite{Nardi2014} the problem $\Delta u =  f + \eps u \in C^{0,\lambda}$ with homogeneous Neumann boundary conditions has a unique solution in $C^{2,\lambda}(\overline\Omega)$, $u$ must itself be $C^{2,\lambda}(\overline{\Omega})$ and satisfy $Du \cdot n = 0$ pointwise on $\del\Omega$. Hence $\mathfrak L_0:\mathfrak A \to \mathfrak B$ is surjective, and we are done.
\end{proof}

\begin{lem}\label{lem:Rm_truncated_flux}
    Let $\theta \in (0,1)$ and let $h(s) \doteq s(1-s^2)^{-1/2}$, so that $A(p) = w_p\,p = h(|p|)\tfrac{p}{|p|}$. There exists $h_\theta \in C^\infty([0,\infty))$ with
    \begin{align*}
        h_\theta = h \ \text{ on } [0, 1-\theta], \qquad
        1 \leq h_\theta'(s) \leq \Mu_\theta
        \doteq \big(\theta(1-\tfrac\theta4)\big)^{-3/2}, \qquad
        \tfrac{h_\theta(s)}{s} \in [1,\Mu_\theta] \quad \forall s>0,
    \end{align*}
    such that the vector field $A_\theta(p) \doteq h_\theta(|p|)\tfrac{p}{|p|}$, $A_\theta(0) = 0$, belongs to $C^\infty(\R^m;\R^m)$, coincides with $A$ on $|p| \leq 1 - \theta$ and satisfies
    \begin{align}\label{eq:Rm_truncated_ellipticity}
        |\xi|^2 \ \leq\ \frac{\del A_\theta^i}{\del p_j}(p)\,\xi_i\xi_j
        \ \leq\ \Mu_\theta\,|\xi|^2,
        \qquad \forall\, p,\xi \in \R^m.
    \end{align}
\end{lem}
\begin{proof}
    The eigenvalues of $DA_\theta(p)$ are $h_\theta'(|p|)$, in the radial direction, and $h_\theta(|p|)/|p|$ with multiplicity $m-1$; for the untruncated field one has $h'(s) = w_s^3$ and $h(s)/s = w_s$, both lying in $[1,\Mu_\theta]$ for $s \leq 1-\theta/2$. It then suffices to extend $h$ beyond $1-\theta$ by a function $h_\theta$ whose derivative interpolates monotonically between $h'(1-\theta)$ and the value $1$, constant for $s \geq 1-\theta/2$: both families of eigenvalues remain in $[1,\Mu_\theta]$ and \eqref{eq:Rm_truncated_ellipticity} follows.
\end{proof}

\begin{thm}[{cf. of \cite[Theorem~2]{Lieberman88}}]\label{thm:Rm_lieberman_adapted}
    Let $\Omega \subset \R^m$ be a bounded domain with $\del\Omega \in C^{1,1}$, let $\theta \in (0,1)$, $M_0, \Lambda_0 \geq 0$, and let $A_\theta$ be the truncated field of \cref{lem:Rm_truncated_flux}. Let $u \in W^{1,2}(\Omega) \cap L^\infty(\Omega)$, with $\norm{u}_{L^\infty(\Omega)} \leq M_0$, be a weak solution of the problem
    \begin{align}\label{eq:Rm_conormal_truncated}
        \begin{cases}
            \div\big(A_\theta(Du)\big) = F &\qquad \text{in $\Omega$} \\[4pt]
            A_\theta(Du)\cdot n = 0 &\qquad \text{on $\del\Omega$,}
        \end{cases}
    \end{align}
    with $F \in L^\infty(\Omega)$, $\norm{F}_{L^\infty(\Omega)} \leq \Lambda_0$.
    Then there exist
    \begin{align*}
        \beta = \beta(m,\theta) \in (0,1), \qquad C_L = C_L(m,\theta,M_0,\Lambda_0,\Omega) > 0
    \end{align*}
    such that $u \in C^{1,\beta}(\overline\Omega)$ and
    \begin{align}\label{eq:Rm_lieberman_estimate}
        \norm{u}_{C^{1,\beta}(\overline\Omega)} \ \leq\ C_L.
    \end{align}
    In particular, if $u \in C^2(\overline\Omega)$ is strictly space-like with $\norm{Du}_{L^\infty(\Omega)} \leq 1-\theta$, $Du \cdot n = 0$ on $\del\Omega$ and mean curvature $H_u \in L^\infty(\Omega)$ with $\norm{H_u}_{L^\infty(\Omega)} \leq \Lambda_0$, then estimate \eqref{eq:Rm_lieberman_estimate} holds with
    $M_0 = \norm{u}_{L^\infty(\Omega)}$.
\end{thm}
\begin{proof}
    Problem \eqref{eq:Rm_conormal_truncated} falls within the scope of \cite[Theorem~2]{Lieberman88} with the following choices. The field is $A(x,z,p) = A_\theta(p)$: the structure conditions (0.3a), (0.3b) in \cite{Lieberman88} hold in the non-degenerate case $m_{\mathrm{Lieb}} = 0$, $\kappa = 1$, with $\lambda = 2^{-1}$ and $\Lambda = \Mu_\theta$ by \eqref{eq:Rm_truncated_ellipticity} (the exponents $(\kappa+|p|)^{0} = 1$); condition (0.3c) is satisfied with $\alpha = 1$ and vanishing constant, since $A_\theta$ does not depend on $(x,z)$. The lower-order term is $B(x,z,p) = -F(x)$, which verifies (0.3d) with constant $\Lambda_0$. The conormal datum is $\phi \equiv 0$, which verifies (0.6) with $\Phi = 0$ (in particular the orientation of the normal is immaterial). The constants $\beta$ and $C_L$ of \cite[Theorem~2]{Lieberman88} then depend only on $(m,\alpha,\Lambda/\lambda) = (m,1,2\Mu_\theta)$, that is, on $(m,\theta)$, and on $(m,\theta,M_0,\Lambda_0,\Omega)$ respectively. 
    
    For the last assertion it suffices to observe that, if $\norm{Du}_{L^\infty(\Omega)} \leq 1-\theta$, then $A_\theta(Du) = A(Du) = w\,Du$ pointwise, so that $u$ is a weak solution of \eqref{eq:Rm_conormal_truncated} with $F = H_u$: the weak formulation is \eqref{eq:Rm_weak_graph} rewritten with respect to the Lebesgue measure, and the conormal condition follows by integrating by parts against $\phi \in C^1(\overline\Omega)$ and using $Du\cdot n = 0$.
\end{proof}

\addcontentsline{toc}{section}{Bibliography}
\printbibliography[title={Bibliography}]

\end{document}